\documentclass[final]{siamltex704}

\usepackage{wrapfig, rotating}
\usepackage{color}
\usepackage{euscript}

\usepackage[utf8]{inputenc} 
\usepackage{textcomp} 
\usepackage{graphicx,epsfig}  
\usepackage{epstopdf} 
\usepackage{flafter}  

\usepackage{amsmath,amssymb}
\usepackage{MnSymbol}
\usepackage{bm}  
\usepackage{enumerate}
\usepackage{subfigure}

\newcommand{\boldN}{\ensuremath{\mathbf{N}}}
\newcommand{\boldNr}{\ensuremath{\widetilde{\mathbf{N}}}}

\newcommand{\bHs }{\ensuremath{\mathbf{H}_k(s_1,\ldots,s_k)}}

\newcommand{\Sys}{\ensuremath{\zeta}}

\newcommand{\boldA}{\ensuremath{\mathbf A}}
\newcommand{\boldAr}{\ensuremath{\widetilde{\mathbf{A}}}}
\newcommand{\boldb}{\ensuremath{\mathbf b}}

\newcommand{\boldbr}{\ensuremath{\widetilde{\mathbf {b}}}}
\newcommand{\boldB}{\ensuremath{\mathbf B}}
\newcommand{\boldC}{\ensuremath{\mathbf C}}
\newcommand{\boldCr}{\ensuremath{\widetilde{\mathbf {C}}}}
\newcommand{\boldBr}{\ensuremath{\widetilde{\mathbf {B}}}}
\newcommand{\boldc}{\ensuremath{\mathbf c}}
\newcommand{\boldcr}{\ensuremath{\widetilde{\mathbf{c}}}}

\newcommand{\boldW}{\ensuremath{\mathbf{W}}}
\newcommand{\boldWr}{\ensuremath{\widetilde{\mathbf{W}}}}
\newcommand{\boldV}{\ensuremath{\mathbf{V}}}
\newcommand{\boldVr}{\ensuremath{\widetilde{\mathbf{V}}}}
\newcommand{\reals}{\ensuremath{\mathbb{R}}}
\newcommand{\complex}{\ensuremath{\mathbb{C}}}

\newcommand{\ud}{\,\mathrm{d}}

\newcommand{\boldU}{\ensuremath{\mathbf{U}}}

\newcommand{\bolde}{\ensuremath{\mathbf{e}}}
\newcommand{\boldI}{\ensuremath{\mathbf{I}}}
\newcommand{\boldIr}{\ensuremath{\mathbf{I}_r}}

\newcommand{\cbfA}{\mbox{\boldmath${\EuScript{A}}$} }
\newcommand{\cbfI}{\mbox{\boldmath${\EuScript{I}}$} }
\newcommand{\cbfN}{\mbox{\boldmath${\EuScript{N}}$} }
\newcommand{\cbfM}{\mbox{\boldmath${\EuScript{M}}$} }
\newcommand{\cbfE}{\mbox{\boldmath${\EuScript{E}}$} }

\newcommand{\cbfAhat}{\widehat{\cbfA\,}}
\newcommand{\cbfAtilde}{\widetilde{\cbfA\,}}
\newcommand{\cbfNhat}{\widehat{\cbfN}}
\newcommand{\cbfNtilde}{\widetilde{\cbfN}}

\newtheorem{thm}{Theorem}[section]
\newtheorem{lem}{Lemma}[section]

\newtheorem{rem}{Remark}[section]

\newtheorem{dfn}{Definition}
\newtheorem{algorithm}{Algorithm}

\title{Multipoint Volterra Series Interpolation and $\mathcal{H}_2$ Optimal Model Reduction of Bilinear Systems\thanks{This work was
supported in part by the NSF Grants DMS-0645347 and DMS-1217156.}}
\author{Garret Flagg, Serkan Gugercin \\ 
{\small Department of Mathematics, Virginia Tech.} \\
{\small Blacksburg, VA, 24061-0123} \\
{\small \tt e-mail: \{flagg,gugercin\}@math.vt.edu}
}

\author{Garret Flagg and Serkan Gugercin
\thanks{Garret Flagg is with Schlumberger/WesternGeco, Houston, TX, 77042, United States; \texttt{gflagg@slb.com}. Serkan Gugercin is with the Department of \mbox{Mathematics}, \mbox{Virginia Tech}, Blacksburg, VA, 24061-0123, USA;  {\texttt{gugercin@math.vt.edu}}.
}
}

\begin{document}
\maketitle

\begin{abstract}
In this paper, we focus on model reduction of large-scale bilinear systems. The main contributions  are threefold. First, we introduce a new framework for interpolatory model reduction of bilinear systems. In contrast to the
existing methods where interpolation is forced on some of the leading subsystem transfer functions, the new framework shows how to enforce multipoint interpolation of the underlying Volterra series.  Then, we show that the first-order conditions for optimal $\mathcal{H}_2$ model reduction of bilinear systems  require multivariate Hermite interpolation in terms of the new Volterra series interpolation framework; and thus we extend the interpolation-based first-order necessary conditions for $\mathcal{H}_2$ optimality of LTI systems  to the bilinear case. 
Finally,  we show that multipoint interpolation on the truncated Volterra series representation of a bilinear system leads to an asymptotically optimal approach to $\mathcal{H}_2$ optimal model reduction, leading to 
an efficient model reduction algorithm. Several numerical examples illustrate the effectiveness of the proposed  approach.
\end{abstract}


\pagestyle{myheadings}
\thispagestyle{plain}
\markboth{Garret Flagg and Serkan Gugercin}{MULTIPOINT VOLTERRA SERIES INTERPOLATION}

\section{Introduction}
Direct numerical simulation of dynamical systems has proven to be a principal tool in modeling, prediction and control of  a wide range of  physical phenomena. However, the growing need for accuracy in the modeling stage leads to very large-scale, complex dynamical systems whose simulations incur a huge burden on computational resources. This motivates {\it model reduction}, whose goal is to accurately approximate large-scale dynamical systems by simpler, smaller ones. These simpler reduced models are then used as surrogates to the original one in prediction, control or optimization settings. The theory and computational tools for model reduction of linear dynamical systems have matured drastically over the last two decades, leading to a greater focus on nonlinear systems.  Bilinear systems, which we consider in this paper, present us with a framework for extending the theory and methodology of model reduction from linear models to nonlinear ones. These models are a special class of (weakly) nonlinear systems characterized by the following systems of ordinary differential equations
\begin{equation}\label{bilinear_eqns}
\zeta:\left \{
\begin{array}{l}
\dot{\mathbf{x}}(t)=\mathbf{A x}(t)+\sum \limits_{k=1}^m\boldN_k\mathbf{x}(t)u_k(t)+\mathbf{B} \mathbf{u}(t)\\
\mathbf{y}(t)=\mathbf{C  x}(t),
\end{array}\right. 
\end{equation}
where $\boldA, \boldN_k \in \reals^{n\times n}$ for $k=1,\dots m$, $\boldB \in \reals^{n\times m}$ and $\boldC \in \reals^{p\times n}$. For brevity of the presentation, at times we will use the notation $\zeta=(\boldA,\boldN_1,\ldots,\boldN_m,\boldB,\boldC)$ to denote the bilinear system of (\ref{bilinear_eqns}).

Bilinear systems arise in a variety of applications ranging from examination of biological species to nuclear fission, and have proven most useful in modeling nonlinear phenomena of small magnitude \cite{Rugh,mohler,mohler1970natural,weiner1980sinusoidal}.  Recently, they have been used as natural models for stochastic control problems \cite{hartmann_balanced}, and have also proven useful in the model reduction of  parameter-varying linear systems \cite{benner2011h2}.   Given a bilinear system $\zeta$ of dimension $n$ as in (\ref{bilinear_eqns}), the goal of  model reduction in this setting is to construct a reduced bilinear system 

\begin{equation}\label{bilinear_eqns_rom}
\widetilde\zeta:\left \{
\begin{array}{l}
\dot{\widetilde{\mathbf{x}}}(t)=\boldAr\widetilde{\mathbf{x}}(t)+\sum \limits_{k=1}^m\boldNr_k\widetilde{\mathbf{x}}(t)u_k(t)+\boldBr \mathbf{u}(t)\\
\widetilde{\mathbf{y}}(t)=\boldCr\widetilde{\mathbf{x}}(t),
\end{array}\right. 
\end{equation}
where $\boldAr, \boldNr_k \in \reals^{r\times r}$ for $k=1,\dots m$, $\boldBr \in \reals^{r\times m}$ and $\boldCr \in \reals^{p\times r}$ with $r \ll n$ such that $\widetilde{\mathbf{y}}(t)$ is an accurate approximation to ${\mathbf{y}}(t)$ in an appropriate norm.
As for the full-order model,  $\widetilde\zeta=(\widetilde\boldA,\widetilde\boldN_1,\ldots,\widetilde\boldN_m\widetilde\boldB,\widetilde\boldC)$ will denote the reduced bilinear system of (\ref{bilinear_eqns_rom}).

As in the linear case, we will construct the reduced model via projection. We will construct two matrices $\boldVr \in \reals^{n \times r}$ and $\boldWr \in \reals^{n \times r}$ such that $\boldWr^T\boldVr$ is invertible. Then, the reduced-model in (\ref{bilinear_eqns_rom}) is given by 
\begin{eqnarray}
\begin{array}{ll}
\boldAr = (\boldWr^T\boldVr)^{-1}\boldWr^T \boldA \boldVr, & \boldNr_k = (\boldWr^T\boldVr)^{-1}\boldWr^T \boldN_k \boldVr ~{\rm for}~k=1,\ldots,m,
\\ 
\boldBr = \boldWr^T \boldB, ~~~~~{\rm and}~~~~~&
\boldCr =  \boldC \boldVr.
\end{array}
\end{eqnarray}
Several approaches from model reduction of linear systems have been already extended to bilinear systems. For example,
bilinear counterparts of gramians have been developed in  \cite{zhang2002h2,benner2011lyapunov}, and gramian-based model reduction techniques such as  balanced truncation  have been proposed in \cite{benner2011lyapunov}. However, as clearly discussed in \cite{benner2011lyapunov}, the resulting generalized Lyapunov equations in the bilinear case present enormous computational challenges even for medium scale problems. Even when these generalized Lyapunov equations can be solved, the resulting reduced models are not guaranteed to attain the nice properties such as an error bound enjoyed by balanced truncation in the linear case. However, we want to emphasize that for the examples where the bilinear counterpart of balanced truncation is applied, the method has performed quite well in practice; see \cite{benner2011lyapunov} for details. Interpolatory model reduction methods have also been successfully extended to bilinear  systems; see, for example, \cite{phillips2003projection,bai2006projection,breiten2009ms,breiten2010krylov}.  In these approaches, $\boldVr$ and $\boldWr$  are chosen in such a way that the subsystem transfer functions of the reduced model interpolates those of the full model at selected interpolation points. 
Zhang and Lam in  \cite{zhang2002h2} were the first to focus on the $\mathcal{H}_2$ optimal approximation of bilinear systems. In \cite{zhang2002h2} they extended the grammian-based Wilson conditions \cite{wilson1970optimum} for $\mathcal{H}_2$ optimality of linear  systems to the bilinear case. However, until very recently it was not clear how to enforce these optimality conditions.
Breiten and Benner in \cite{Breiten_H2} re-formulated these conditions in an equivalent but numerically more effective framework and  showed how to achieve $\mathcal{H}_2$ optimal approximations via an iterative, projection-based approach, called the Bilinear Iterative Rational Krylov Algorithm (B-IRKA), and thus extended the Iterative Rational Krylov Algorithm 
(IRKA) of Gugercin {\it et al.} \cite{H2} from  linear systems to  bilinear ones. B-IRKA has proved very successful, leading to  high-fidelity reduced models and  outperforming balancing-based bilinear model reduction methods, and has become the method of choice in most cases.

In this paper, we  focus on interpolatory approaches for reducing bilinear systems. The main contributions are threefold. After giving a short background on bilinear systems and a new derivation for the $\mathcal{H}_2$ norm of a bilinear system in Section \ref{siso_h2}, in  Section \ref{sec:volterra} we  introduce a new framework for interpolatory model reduction of bilinear systems where we show how to enforce multipoint interpolation of the underlying  Volterra series. This is in contrast to the current techniques where interpolation is enforced only on subsystem transfer functions as opposed to the Volterra Series. Then, in Section \ref{sec:volterra_h2}, we show that 
this new interpolation framework is indeed what lies behind the optimal $\mathcal{H}_2$ approximation of bilinear systems; 
thus generalizing the interpolation-based first-order necessary conditions for $\mathcal{H}_2$ optimality of LTI systems presented in \cite{H2} to the bilinear case.  Finally, in Section \ref{sec:TBIRKA} we show that multipoint interpolation on the the truncated Volterra series representation of a bilinear system leads to an asymptotically optimal approach to $\mathcal{H}_2$ optimal model reduction which is inexpensive to implement. Section \ref{sec:examples} illustrates the theoretical discussions via several numerical examples followed by conclusions in Section \ref{sec:conc}.

As noted in (\ref{bilinear_eqns}), we are interested in reducing multi-input/multi-output  
(MIMO) bilinear systems,
 and one of the main contributions of this paper, as presented in Algorithm \ref{alg:tbrika}.
However, this algorithm was inspired by an analysis of the interpolation properties associated with SISO bilinear systems, the other main contribution of the paper.  The interpolation-based approach to model reduction of bilinear SISO systems provides insight into the properties of $\mathcal{H}_2$ optimal bilinear approximations, but the results of the analysis are not readily generalizable to the MIMO case in their current formulation.  As such, our analysis of Volterra series type interpolation constraints is specific to SISO systems at a formal level, but is the basis for our approach to $\mathcal{H}_2$ optimal model reduction of both SISO and MIMO bilinear systems. Thus,  wherever we focus on SISO systems, this is clearly stated. For SISO bilinear systems, we will use the notation $\zeta=(\boldA,\boldN,\boldb,\boldc)$ where 
$\boldb,\boldc^T \in \reals^n$. 

\section{Background}
The external representation of a causal, stationary bilinear system $\zeta$ admits the following Volterra series representation which describes the nonlinear mapping of admissible inputs $\mathbf{u}(t) \in \mathcal{U} \subset \reals^m$ to outputs $\mathbf{y}(t)\in \reals^p$:

 \begin{equation}
\mathbf{y}(t)=\sum_{k=1}^{\infty}\int_0^{t_1}\int_0^{t_k}\cdots \int_0^{t_{k}}\mathbf{h}_k(t_1,t_2,\dots,t_k)(\mathbf{u}(t-\sum_{i=1}^k t_i)\otimes \cdots \otimes \mathbf{u}(t-t_k))\ud t_k \cdots \ud t_1.
\end{equation}
The regular Volterra kernels $\mathbf{h}_k$ are given as
\begin{align}
\mathbf{h}_k(t_1,t_2,\dots,t_k)=&\boldC e^{\boldA t_k}\bar{\boldN}(\boldI_m\otimes e^{\boldA t_{k-1}})(\boldI_m \otimes \bar{\boldN}) \cdots \nonumber \\ 
&(\underbrace{\boldI_m\otimes \cdots \otimes \boldI_m}_{k-2 \text{ times}} \otimes e^{\boldA t_2})(\underbrace{\boldI_m \otimes \cdots \otimes \boldI_m}_{k-2 \text{ times }} \otimes \bar{\boldN})\label{mimo_kernel}\\
&\cdot(\underbrace{\boldI_m \otimes \cdots \otimes \boldI_m}_{k-1 \text{ times}} \otimes e^{\boldA t_1})(\underbrace{\boldI_m \otimes \cdots \otimes \boldI_m}_{k-1 \text{ times }} \otimes \boldB)\nonumber.
\end{align}
where  $\bar{\boldN}=\begin{bmatrix}\boldN_1, \dots, \boldN_m\end{bmatrix}$.

The multivariable Laplace transform of the degree $k$ regular kernel (\ref{mimo_kernel}) of \Sys\ is given by
\begin{align}
\bHs=&\boldC(s_k\boldI-\boldA)^{-1}\bar{\boldN}[\boldI_m \otimes (s_{k-1}\boldI-\boldA)^{-1}](\boldI_m \otimes \bar{\boldN}) \cdots \nonumber \\
&\cdot[\underbrace{\boldI_m \otimes \cdots \otimes \boldI_m}_{k-2 \text{ times}} \otimes (s_2\boldI-\boldA)^{-1}](\underbrace{\boldI_m \otimes \cdots \otimes \boldI_m}_{k-2 \text{ times }} \otimes \bar{\boldN}) \label{MIMO_TF}\\
&\cdot[\underbrace{\boldI_m \otimes \cdots  \otimes \boldI_m}_{k-1 \text{ times}} \otimes (s_1\boldI-\boldA)^{-1}](\underbrace{\boldI_m \otimes \cdots \otimes \boldI_m}_{k-1 \text{ times }} \otimes \boldB)\nonumber.
\end{align}  
The functions $\bHs$ are called the $kth$ order transfer functions of the bilinear system.  

\subsection{The $\mathcal{H}_2$ norm}
As in the case of model reduction of linear systems, we need an appropriate  measure to quantify the error induced by the  reduction process. In this paper, we will focus on the $\mathcal{H}_2$ norm:
\begin{dfn}
 Let $\zeta$ be a MIMO bilinear system.  Define the $\mathcal{H}_2$ norm of $\zeta$ as
\begin{equation}
\|\Sys\|_{ _{\mathcal{H}_2}}=\Big(\sum \limits_{k=1}^{\infty} \sup \limits_{x_1>0, \dots, x_k>0} \int\limits_{-\infty}^{\infty}\cdots\int \limits_{- \infty}^{\infty} \big\|\mathbf{H}_i(x_1+\imath y_1,\dots, x_k+\imath y_1)\big\|_F^2 \mathrm{d}y_1\cdots \mathrm{d}y_k \Big)^{1/2},
\end{equation}
where $\mathbf{H}_i(s_1,\dots, s_k)$ is the $kth$ order transfer function as in (\ref{MIMO_TF}) and
 $\|\mathbf{H}_i(s_1,\dots, s_k)\|_F$ denotes the Frobenius norm of $\mathbf{H}_k(s_1, \dots, s_k)$.     
\end{dfn}

From Plancheral's theorem in several variables, the $\mathcal{H}_2$ norm of the bilinear system is equivalent to its $\mathcal{L}_2$ norm in the time domain. The $\mathcal{L}_2$ norm of a bilinear system is 

\begin{equation*}
\|\Sys\|_{\mathcal{L}_2}=\Big(\sum \limits_{k=1}^{\infty}\int_{0}^{\infty}\cdots\int_{0}^{\infty}\|\mathbf{h}_k(t_1,\dots,t_k)\|_F^2 \mathrm{d}t_1 \cdots \mathrm{d}t_k\Big)^{1/2}.
\end{equation*} 
where $\mathbf{h}_k(t_1,\dots,t_k)$ is the $kth$ Volterra kernel given in (\ref{mimo_kernel}).
When this norm converges, Zhang and Lam showed in \cite{zhang2002h2} that it can be given in terms of the realization parameters \boldA, $\boldN_1,\dots, \boldN_m$, \boldB, \boldC, as
\begin{equation}
\big\|\Sys\big\|_{\mathcal{L}_2}=\sqrt{trace(\boldC\sum \limits_{k=0}^{\infty} \mathbf{P}_k \boldC^T)}=\big\|\Sys\big\|_{\mathcal{H}_2}, \label{norm_equivalence}
\end{equation}
where $\mathbf{P}_0$ solves
\begin{equation}
\boldA \mathbf{P}_0+\mathbf{P}_0\boldA^T=\boldB\boldB^T,
\end{equation}
and for $k\ge1$, $\mathbf{P}_k$ solves
\begin{equation}
\boldA \mathbf{P}_k+\mathbf{P}_k\boldA^*=\sum \limits_{j=1}^M \boldN_j \mathbf{P}_{k-1} \boldN_j^T.
\end{equation}

\subsection{A pole-residue formulation of the $\mathcal{H}_2$ norm for SISO systems} \label{siso_h2}
In this section, we give a rigorous generalization of the pole-residue formula for the standard Hardy space $\mathcal{H}_2$ norm of a SISO LTI system to the case of SISO bilinear systems.   A similar expression was given independently by Breiten and Benner in \cite{benner2010krylov}, though our derivation of it here is new.
 These pole-residue based expressions will be used as the motivation for our multipoint interpolation method in Section \ref{sec:volterra}. 

In the single-input-single-output (SISO) case, the $kth$ order transfer functions of a bilinear system given in (\ref{MIMO_TF}) reduce to 
\begin{equation}
H_k(s_1,s_2,\dots,s_k)=\boldc(s_k\boldI-\boldA)^{-1}\boldN(s_{n-1}\boldI-\boldA)^{-1}\boldN\cdots\boldN(s_1\boldI-\boldA)^{-1}\boldb.
\end{equation}
Writing $(s_i\boldI-\boldA)^{-1}$ as the classical adjoint over the determinant, it is readily seen that 
\begin{equation}
H_k(s_1,s_2,\dots, s_k)=\frac{P(s_1,s_2,\dots,s_k)}{Q(s_1)Q(s_2)\cdots Q(s_k)},
\end{equation}
where $Q(s_j)=det(s_j\boldI-\boldA)$ for $j=1,\dots, k$, and $P(s_1,s_2,\dots, s_k)$ is a $k$-variate polynomial with maximum degree $k(n-1)$.  Thus $H_k(s_1,s_2,\dots,s_k)$ is a proper $k$-variate rational function whose singularities are characterized by a simple analytic variety.  This allows us to give a straightforward definition of the residues of $H_k(s_1,\dots, s_k)$, and write it as a sum of partial fractions determined by the residues and poles, analgous to the single variable case.
\begin{dfn}\label{dfn:residues}
For a $kth$-order transfer function $H_k(s_1,\dots, s_k)$, define the residues of $H_k(s_1,\dots, s_k)$ as
\begin{equation}
\phi_{ _{l_1,\dots,l_k}}=\lim\limits_{s_k \rightarrow \lambda_{l_k}}(s_k-\lambda_{l_k})\lim\limits_{s_{k-1} \rightarrow \lambda_{l_{k-1}}}(s_{k-1}-\lambda_{l_{k-1}})\cdots\lim\limits_{s_1\rightarrow \lambda_{l_1}}(s_1-\lambda_{l_1})H_k(s_1, \dots, s_k).
\end{equation}
\end{dfn}

\begin{thm}[Pole-Residue Formula for $H_k(s_1,\dots, s_k)$]\label{thm:pole_residue}\hspace{2pt}\\
Let $\displaystyle{H_k(s_1, \dots, s_k)=\frac{P(s_1,\dots,s_k)}{Q(s_1)Q(s_2)\cdots Q(s_k)}}$ where $P(s_1,\dots,s_k)$ is a polynomial in $k$ variables of total degree $k(n-1)$ and $Q(s_i)$ is a polynomial of degree $n$ in the variable $s_i$ with simple zeros at the points $\lambda_1, \dots, \lambda_n \in \complex$.  Then
\begin{equation}
H_k(s_1, \dots, s_k)=\sum\limits_{l_1=1}^n\cdots\sum \limits_{l_k=1}^n \frac{\phi_{ _{l_1,\dots,l_k}}}{\prod \limits_{i=1}^k(s_i-\lambda_{l_i})}
\end{equation}
\end{thm}

\begin{proof}
For the brevity of the paper, we skip the proof of this paper and refer the reader 
to \cite{flagg2012interpolation}.
\end{proof}

This pole-residue formula will be used next to derive an expression for the $\mathcal{H}_2$ norm of a SISO bilinear system.

\begin{thm}[$\mathcal{H}_2$ norm expression]\label{thm:H2_norm_expression}
Let $\zeta$ be a SISO bilinear system with a finite $\mathcal{H}_2$ norm.  Then 
$$\|\zeta\|^2_{ _{\mathcal{H}_2}}=\sum \limits_{k=1}^{\infty}\sum \limits_{l_1=1}^{n}\sum\limits_{l_2=1}^{n}\cdots\sum\limits_{l_k=1}^n \phi_{ _{l_1,\dots,l_k}}H_k(-\lambda_{l_1},\dots,-\lambda_{l_k}).$$
\end{thm} 

\begin{proof}
 From Plancherel's Theorem,
\begin{align}
\|\zeta\|^2_{ _{\mathcal{H}_2}}&=\|\zeta\|^2_{ _{\mathcal{L}_2(\imath\reals)}}\nonumber\\
&=\sum \limits_{k=1}^{\infty}\frac{1}{(2\pi)^k}\int \limits_{- \infty}^{\infty}\cdots \int \limits_{-\infty}^{\infty} H_k(-\imath \omega_1,\dots,-\imath\omega_k)H_k(\imath \omega_1, \dots, \imath \omega_k)\mathrm{d} \omega_k \mathrm{d}\omega_{k-1} \cdots \mathrm{d}\omega_{1}, \label{integral_expression}
\end{align}
and from Theorem \ref{thm:pole_residue} 
\begin{equation}\label{pole_residue}
H_k(s_1, \dots,s_k)=\sum \limits_{l_1=1}^n\cdots \sum \limits_{l_k=1}^n \frac{\phi_{l_1, \dots,l_k}}{\prod \limits_{i=1}^k(s_i-\lambda_{l_i})}.
\end{equation}

Substituting (\ref{pole_residue}) for $H_k(\imath\omega_1,\dots,\imath \omega_k)$ at the $kth$ term in the series (\ref{integral_expression}) and considering this term alone gives
\begin{align}
&=\frac{1}{(2\pi)^k}\int \limits_{- \infty}^{\infty}\cdots \int \limits_{- \infty}^{\imath \infty}\sum\limits_{l_1=1}^{n}\cdots\sum\limits_{l_k=1}^{n} \frac{\phi_{l_1,\dots,l_k}H_k(-\imath\omega_1,\dots,-\imath\omega_k)} {\prod \limits_{i=1}^k(\imath \omega_i-\lambda_{l_i})} \mathrm{d}\omega_k \mathrm{d}\omega_{k-1} \cdots \mathrm{d}\omega_{1}\nonumber\\
&=\sum\limits_{l_1=1}^{n}\cdots\sum\limits_{l_k=1}^{n}\frac{1}{(2\pi)^k}\int \limits_{-\infty}^{ \infty}\cdots \int \limits_{- \infty}^{\infty}
\frac{\phi_{l_1,\dots,l_k}H_k(-\imath\omega_1,\dots,-\imath \omega_k)} {\prod \limits_{i=1}^k(\imath \omega_i-\lambda_{l_i})} \mathrm{d}\omega_k \mathrm{d}\omega_{k-1} \cdots \mathrm{d}\omega_{1}\nonumber\\
&=\sum\limits_{l_1=1}^{n}\cdots\sum\limits_{l_k=1}^{n}
\phi_{l_1,\dots,l_k}H_k(-\lambda_{l_1},\dots,-\lambda_{l_k}) \label{last_step_1}
\end{align}
The expression in (\ref{last_step_1}) is an application of Cauchy's formula in $k$-variables, in the following way.  Consider the contours $\gamma_{R_j}=[-\imath R_j, \imath R_j] \cup \{z= R_je^{\imath \theta} \text{ for } \pi/2 \le \theta \le \frac{3\pi}{2}\}$ for $j=1,\dots, k$ in the complex plane, and let $\Gamma=\bigtimes\limits_{j=1}^k\gamma_{R_j}$ be the distinguished boundary of the polycylinder given by the set of points $\mathcal{D}_{R_1,\dots, R_j}=\{ (s_1, \dots, s_k)| s_j \in \text{int} \gamma_{R_j} \text{ for } j=1,\dots k\}$, where ``int'' denotes the interior of the contour.  For all sufficiently large $R_j$, $j=1,\dots, k$ all the points $(\lambda_{l_1}, \dots, \lambda_{l_k}) \in \mathcal{D}_{R_1, \dots, R_k}$ for $l_1,\dots, l_k =1, \dots,n$.  But the functions $H_k(-s_1,\dots, -s_k)$ are holomorphic on $\mathcal{D}_R$, and so by Cauchy's formula (see \cite{rudin1969function} for details on extending Cauchy's formula to polycylinders)
\begin{align*}
H_k(-\lambda_{l_1}, \dots, -\lambda_{l_k})&=\frac{1}{(2\pi \imath)^k}\int \limits_{\gamma_{R_1}}\cdots \int \limits_{\gamma_{R_k}}
\frac{H_k(-s_1,\dots,-s_k)} {\prod \limits_{i=1}^k(s_i-\lambda_{l_i})} \mathrm{d}s_k \mathrm{d}s_{k-1} \cdots \mathrm{d}s_{1}\\
&=\frac{1}{(2\pi\imath)^k}\int \limits_{\gamma_{R_1}}\cdots\int\limits_{\gamma_{R_2}}( \int\limits_{\pi/2}^{3\pi/2}
-\imath\frac{H_k(-s1,\dots,-R_k e^{-\imath \theta_k}) R_k\imath e^{\imath \theta_k}} {\prod \limits_{i=1}^{k-1}(s_i-\lambda_{l_i}) (R_ke^{\imath\theta_k}-\lambda_{l_k})} \mathrm{d}\theta_k \\
&+\int \limits_{-R_k}^{R_k} \frac{\imath H(-s_1,\dots, -\imath\omega_k)}{\prod \limits_{i=1}^{k-1}(s_i-\lambda_{l_i}) (\imath \omega_k-\lambda_{l_k}) }\mathrm{d}\omega_k )\mathrm{d}s_{k-1} \cdots \mathrm{d}s_{1}
\end{align*} 
Letting $R_k \rightarrow \infty$, the term \[|\int\limits_{\pi/2}^{3\pi/2}
-\imath\frac{H_k(-s1,\dots,-R_k e^{-\imath \theta_k}) R_k\imath e^{\imath \theta_k}} {\prod \limits_{i=1}^{k-1}(s_i-\lambda_{l_i}) (R_ke^{\imath\theta_k}-\lambda_{l_k})} \mathrm{d}\theta_k| \rightarrow 0,\] since $H_k(-s_1, \dots, -s_k)$ is a proper rational function in the variable $s_k$.  Thus, 
\begin{align*}
H_k(-\lambda_{l_1}, \dots, -\lambda_{l_k})&=\frac{1}{(2\pi \imath)^k}\int \limits_{\gamma_{R_1}}\cdots \int \limits_{\gamma_{R_k}}
\frac{H_k(-s_1,\dots,-s_k)} {\prod \limits_{i=1}^k(s_i-\lambda_{l_i})} \mathrm{d}s_k \mathrm{d}s_{k-1} \cdots \mathrm{d}s_{1}\\
&=\frac{1}{(2\pi\imath)^k}\int \limits_{\gamma_{R_1}}\cdots\int\limits_{\gamma_{R_2}}\int \limits_{-\infty}^{\infty} \frac{\imath H(-s_1,\dots, -\imath\omega_k)}{\prod \limits_{i=1}^{k-1}(s_i-\lambda_{l_i}) (\imath \omega_k-\lambda_{l_k}) }\mathrm{d}\omega_k )\mathrm{d}s_{k-1} \cdots \mathrm{d}s_{1}.
\end{align*}
Repeating this argument $k-1$ times yields the desired result that 
\begin{align}
H_k(-\lambda_{l_1}, \dots, -\lambda_{l_k})=\frac{1}{(2\pi)^k}\int \limits_{-\infty}^{ \infty}\cdots \int \limits_{- \infty}^{\infty}
\frac{H_k(-\imath\omega_1,\dots,-\imath \omega_k)} {\prod \limits_{i=1}^k(\imath \omega_i-\lambda_{l_i})} \mathrm{d}\omega_k \mathrm{d}\omega_{k-1} \cdots \mathrm{d}\omega_{1}.
\end{align}
Since this holds for every $k$, returning to our original goal we now have that 
\begin{align*}
&\sum \limits_{k=1}^{\infty}\frac{1}{(2\pi)^k}\int \limits_{- \infty}^{\infty}\cdots \int \limits_{-\infty}^{\infty} H_k(-\imath \omega_1,\dots,-\imath\omega_k)H_k(\imath \omega_1, \dots, \imath \omega_k)\mathrm{d} \omega_k \mathrm{d}\omega_{k-1} \cdots \mathrm{d}\omega_{1} \\
&=\sum \limits_{k=1}^{\infty}\sum\limits_{l_1=1}^{n}\cdots\sum\limits_{l_k=1}^{n}
\phi_{l_1,\dots,l_k}H_k(-\lambda_{l_1},\dots,-\lambda_{l_k}),
\end{align*}
which concludes the proof.
\end{proof}

When the realization term $\boldN$ is zero (so that the system is linear), this expression reduces to the pole-residue expression for the $\mathcal{H}_2$ norm for LTI systems derived  in \cite{gugercinprojection} and \cite{antoulas2005approximation}.

 Now let $\widetilde{\zeta}$ be an $r$-dimensional approximation to an $n$-dimensional bilinear system $\zeta$, with $r \ll n$, and let all reduced-dimension quantities be distinguished by tildes.  Applying the above derivation of the $\mathcal{H}_2$ norm to the error system $\zeta-\widetilde{\zeta}$ yields the following expression for the $\mathcal{H}_2$ error:
\begin{align*}
\|\zeta-\widetilde{\zeta}\|^2_{ _{\mathcal{H}_2}}&=\sum \limits_{k=1}^{\infty}\sum \limits_{l_1=1}^{n}\sum\limits_{l_2=1}^{n}\cdots\sum\limits_{l_k=1}^n \phi_{ _{l_1,\dots,l_k}}(H_k(-\lambda_{l_1},\dots,-\lambda_{l_k})-\widetilde{H}_k(-\lambda_{l_1},\dots,-\lambda_{l_k}))\\
&+\sum \limits_{k=1}^{\infty}\sum \limits_{l_1=1}^{r}\sum\limits_{l_2=1}^{r}\cdots\sum\limits_{l_k=1}^r \widetilde{\phi}_{ _{l_1,\dots,l_k}}(H_k(-\widetilde{\lambda}_{l_1},\dots,-\widetilde{\lambda}_{l_k})-\widetilde{H}_k(-\widetilde{\lambda}_{l_1},\dots,-\widetilde{\lambda}_{l_k})).
\end{align*}
Thus, the $\mathcal{H}_2$-norm error is due to the mismatch of weighted sums of the transfer functions evaluated at all possible combinations.  By analogy with the linear case, we would like to eliminate the error due to the mismatch at the reduced dimension singularities.  In the next section, we introduce a new interpolation scheme that makes it possible to match the full and reduced dimension systems along weighted sums of the transfer functions evaluated at all possible combinations of a collection of frequencies.

\section{Multipoint Volterra series interpolation} \label{sec:volterra}
In this section, we present a new method of multipoint interpolation that respects the external Volterra series representation of bilinear dynamical systems, in the sense that it aims to capture the response of the whole Volterra series with respect to a collection of frequencies.  Similar interpolation problems involving matching a functional defined by weighted sums of function evaluations have appeared in other 
contexts under names such as ``integral interpolation'' \cite{beatson2007integral}. The interpolation framework we introduce here is fundamentally different from the existing interpolation based model reduction approaches for  bilinear systems; such as \cite{phillips2003projection,bai2006projection,breiten2009ms,breiten2010krylov}.  In these works, the goal is to find a reduced bilinear system as in $\widetilde \zeta$ whose leading $kth$ order transfer functions interpolates those of the original one; i.e. 
$$
H_k(\sigma_1,\ldots,\sigma_k) = \widetilde{H}_k(\sigma_1,\ldots,\sigma_k),~~{\rm for}~~
k=1,\ldots,N,
$$
where $\{\sigma_i\}\in \complex$ are the interpolation points in the complex plane. However, instead of interpolating some of the leading subsystem transfer functions, in this paper we will show how to interpolate the whole Volterra series.

  Theorem \ref{thm:pole_residue} indicates that important system properties are measured by weighted sums of the $kth$ order transfer functions evaluated at all possible combinations of the points $-\lambda(\boldA)$.  In general then, we would like to construct reduced-dimensional models that capture these properties of the Volterra series of the full-dimensional system \Sys.  Consider therefore, the following multipoint interpolation problem.  

Given two sets of points $\sigma_1, \sigma_2, \dots, \sigma_r \in \complex$ and $\mu_1,\dots, \mu_r \in \complex$, together with two matrices $\mathbf{U}, \mathbf{S} \in \reals^{r \times  r}$, fix some $j \in \{1, 2, \dots, r\}$ and define the weighted series
\begin{align*}
\nu_j=\sum \limits_{k=1}^{\infty} \sum\limits_{l_1}^r \sum \limits_{l_2}^r \cdots \sum \limits_{l_{k-1}}^r \eta_{ _{l_1,l_2, \dots, l_{k-1},j}}H(\sigma_{l_1},\sigma_{l_2}, \dots, \sigma_j)<\infty\\
\gamma_j=\sum \limits_{k=1}^{\infty} \sum\limits_{l_1}^r \sum \limits_{l_2}^r \cdots \sum \limits_{l_{k-1}}^r \widehat{\eta}_{ _{l_1,l_2, \dots, l_{k-1},j}}H( \mu_j, \mu_{l_2}, \dots, \mu_{l_k}) < \infty
\end{align*}
where $l_1, l_2, \dots, l_k= 1, \dots, r$.  The weights $\eta_{ _{l_1,l_2, \dots, l_{k-1},j}}$ are given in terms of the entries of $\mathbf{U}=\{u_{i,j}\}$ as
\begin{equation}\label{weights}
\eta_{ _{l_1,l_2, \dots, l_{k-1},j}}=u_{j,l_{k-1}}u_{l_{k-1},l_{k-2}}\cdots u_{l_2,l_1} \text{ for } k\ge 2\text{ and }\eta_{l_1}=1 \text{ for }l_1=1, \dots, r.
\end{equation}
 For example, $\eta_{1,2,3}=u_{3,2}u_{2,1}$.  Thus, the weights $\eta_{ _{l_1,l_2, \dots, l_{k-1},j}}$ are generated by multiplying sequences of the entries of $\mathbf{U}$ together in the combinations determined by the index $l_i$. The weights $\widehat{\eta}_{ _{l_1,l_2, \dots, l_{k-1},j}}$ are defined in the same way in terms of the entries of $\mathbf{S}$. Note that for the interpolation conditions in $\sigma_j$,  $s_k=\sigma_j$, whereas $s_1, \dots, s_{k-1}$ may take any other value $\sigma \in \{\sigma_1, \dots, \sigma_r\}$ for all the transfer function evaluations in the series.  Analogously, for the interpolation conditions in $\mu_j$, $s_1=\mu_j$, whereas $s_2 \dots, s_{k-1}$ may take any other value $\mu \in \{\mu_1, \dots, \mu_r\}$ for all the transfer function evaluations in the series.  Given the full-order SISO bilinear system
 ${\zeta}:=$(\boldA, \boldN, \boldb, \boldc), together with the interpolation 
  data $\{\nu_j\}_{j=1}^r, \{\gamma_j\}_{j=1}^r$, the goal is to construct a reduced order system $\widetilde{\zeta}:=$(\boldAr, \boldNr, \boldbr, \boldcr)  of dimension $r$ so that for each $j=1, \dots r$

\begin{align} \label{volterra_series_interpolation}
\sum \limits_{k=1}^{\infty} \sum\limits_{l_1}^r \sum \limits_{l_2}^r \cdots \sum \limits_{l_{k-1}}^r \eta_{ _{l_1,l_2, \dots, l_{k-1},j}}\widetilde{H}_k(\sigma_{l_1},\sigma_{l_2}, \dots, \sigma_j)
=\nu_j 
\end{align}
and 
\begin{equation} \label{volterra_series_interpolation_mu}
\sum \limits_{k=1}^{\infty} \sum\limits_{l_2}^r \sum \limits_{l_3}^r \cdots \sum \limits_{l_{k}}^r \widehat{\eta}_{ _{l_1,l_2, \dots, l_{k-1},j}}\widetilde{H}_k(\mu_j,\mu_{l_2}, \dots, \mu_{l_k})=\gamma_j.
\end{equation}
In order to solve this problem we first consider the special connection between Volterra series and generalized Sylvester equations in the following lemma.

\begin{lemma}\label{lemma:sylvester_equation_solution}
Let $\zeta :=$(\boldA, \boldN, \boldb, \boldc) be a stable SISO bilinear system of dimension $n$. Suppose that for some $r < n$, points $\sigma_1, \dots, \sigma_r \in \complex$ and $\mu_1,\dots,\mu_r  \in \complex$, together with $\mathbf{U}, \mathbf{S} \in \reals^{r \times r}$ are given so that the series

\begin{align} 
\mathbf {v}_j=\sum \limits_{k=1}^{\infty}& \sum\limits_{l_1}^r  \cdots \sum \limits_{l_{k-1}}^r \eta_{ _{l_1, \dots, l_{k-1},j}}(\sigma_j\boldI-\boldA)^{-1}\boldN(\sigma_{l_{k-1}}\boldI-\boldA)^{-1}\boldN \cdots \boldN(\sigma_{l_1}\boldI-\boldA)^{-1}\boldb  \label{sylvester_series}
\end{align}
and
\begin{align}
\mathbf{u}_j =\sum \limits_{k=1}^{\infty} &\sum\limits_{l_1}^r 
\cdots \sum 
\limits_{l_{k-1}}^r \widehat{\eta}_{ _{l_1, \dots, l_{k-1},j}}(\mu_j\boldI-\boldA^T)^{-1}\boldN^T(\mu_{l_{k-1}}\boldI-\boldA^T)^{-1}\boldN^T \cdots \boldN^T(\mu_{l_1}\boldI-\boldA^T)^{-1}\boldc^T \label{sylvester_series_W}
\end{align}
 converge for each $\sigma_j$, and $\mu_j$.   Let $\mathbf{\Lambda}=\text{diag}(\sigma_1, \dots, \sigma_r)$ and $\mathbf{M}=\text{diag}(\mu_1,\dots, \mu_r)$. 
Then the matrices $\boldVr=[\mathbf{v}_1,\dots,\mathbf{v}_r]$, and $\boldWr=[\mathbf{w}_1,\dots,\mathbf{w}_r]\in \reals^{n \times r}$ solve the generalized Sylvester equations

\begin{equation}
\boldVr \mathbf{\Lambda}-\boldA\boldVr-\boldN\boldVr\mathbf{U}^T=\boldb\bolde^T  \label{volt_sylvester_eqn1}
\end{equation}
and
\begin{equation}
\boldWr\mathbf{M}-\boldA^T\boldWr-\boldN^T\boldWr\mathbf{S}^T=\boldc^T\bolde^T. \label{volt_sylvester_eqn2}
\end{equation}

\end{lemma}

\begin{proof}
We first show that the $jth$ column of \boldVr\ is equivalent to (\ref{sylvester_series}).  Let $\boldV^{(1)}\in \reals^{n\times r}$ solve
\begin{equation}
\boldV^{(1)}\mathbf{\Lambda}-\boldA\boldV^{(1)}=\boldb\bolde^T
\end{equation}
and for $k\ge 2$, let $\boldV^{(k)} \in \reals^{n\times r}$ be the solution to 
\begin{equation}
\boldV^{(k)}\mathbf{\Lambda}-\boldA\boldV^{(k)}=\boldN\boldV^{(k-1)}\mathbf{U}^T \label{volt_sylvester_terms}
\end{equation}
Then $\mathbf{v}_{1,j}=(\sigma_j\boldI-\boldA)^{-1}\boldb$, and in general for $k\ge 2$
\begin{equation}
\mathbf{v}_{k,j}=(\sigma_j\boldI-\boldA)^{-1}\mathbf{f}_{k-1,j}
\end{equation}
where $\mathbf{f}_{k-1,j}$ is the $jth$ column of $\boldN\boldV^{(k-1)}\boldU$.  We show  by induction on $k$ that 
$$
\mathbf{f}_{k-1,j}=\sum\limits_{l_{k-1}}^r\sum\limits_{l_{k-2}=1}^r \cdots \sum \limits_{l_1}^r\eta_{ _{l_1,l_2,\dots,l_{k-1},j}}\boldN(\sigma_{l_{k-1}}\boldI-\boldA)^{1}\boldN(\sigma_{l_{k-2}}\boldI-\boldA)^{-1}\boldN\cdots\boldN(\sigma_{l_1}\boldI-\boldA)^{-1}\boldb.
$$
Let $k=2$.  Then 
$$
\mathbf{f}_{1,j}=\sum_{l_1=1}^r u_{ _{j,l_1}}\boldN\mathbf{ v}_{ _{1,l_1}} =\sum_{l_1=1}^r \eta_{ _{l_1,j}}\boldN(\sigma_{l_1}\boldI-\boldA)^{-1}\boldb.
$$
Now suppose the statement holds for $k>2$.  Then 
\begin{align*}
\mathbf{v}_{k,j}&=(\sigma_j\boldI-\boldA)^{-1}\mathbf{f}_{k-1}\\
&=\sum \limits_{l_{k-1}}^r \sum \limits_{l_{k-2}=1}^r \cdots \sum \limits_{l_1=1}^r\eta_{ _{l_1,l_2,\dots,l_{k-1},j}}(\sigma_j\boldI-\boldA)^{-1}\boldN(\sigma_{l_{k-1}}\boldI-\boldA)^{-1}\boldN\cdots(\sigma_{l_1}\boldI-\boldA)^{-1}\boldb,
\end{align*}
and therefore
\begin{align*}
\mathbf{f}_{k,j}&=\sum_{l_{k}}^{r} u_{j,l_{k}}\boldN\mathbf{v}_{k, l_k}\\
&=\sum\limits_{l_k}^r\sum\limits_{l_{k-1}}^r\cdots\sum\limits_{l_1}^r u_{j,l_k}\eta_{l_1,l_2,\dots,l_{k-1},l_k}\boldN(\sigma_{l_k}\boldI-\boldA)^{-1}\boldN\cdots \boldN (\sigma_{l_2}\boldI-\boldA)^{-1}\boldb\\
&=\sum\limits_{l_k}^r\sum\limits_{l_{k-1}}^r\cdots\sum\limits_{l_1}^r \eta_{l_1,l_2,\dots,l_{k-1},l_k,j}\boldN(\sigma_{l_k}\boldI-\boldA)^{-1}\boldN\cdots \boldN (\sigma_{l_2}\boldI-\boldA)^{-1}\boldb.
\end{align*}
By assumption, we therefore have that the series $\boldVr=\sum \limits_{k=1}^{\infty} \mathbf{V}^{(k)}$ converges. Moreover, one can now simply check that $\boldVr$ is a solution to  (\ref{volt_sylvester_eqn1}).  An exactly analogous proof shows that \boldWr\ solves
(\ref{volt_sylvester_eqn2}). 
\end{proof}

\begin{thm}[Volterra Series Interpolation]\label{thm:volt_series_interp}
Let $\zeta :=$(\boldA, \boldN, \boldb, \boldc) be a SISO bilinear system of dimension $n$. Suppose that for some $r < n$, points $\sigma_1, \dots, \sigma_r \in \complex$ and $\mu_1,\dots,\mu_r  \in \complex$, together with $\mathbf{U}, \mathbf{S} \in \reals^{r \times r}$, all the hypotheses of Lemma \ref{lemma:sylvester_equation_solution} hold. Moreover, let $\boldVr$ and $\boldWr$ be the  solutions of (\ref{volt_sylvester_eqn1}) and (\ref{volt_sylvester_eqn2}) respectively as in Lemma \ref{lemma:sylvester_equation_solution}). 
If $\boldWr^T\boldVr \in \reals^{r \times r}$ is invertible, then the reduced order model $\widetilde{\zeta}:=$(\boldAr, \boldNr, \boldbr, \boldcr) of order $r$ defined by 
\begin{align}
\boldAr&=(\boldWr^T\boldVr)^{-1}\boldWr^T\boldA\boldVr,\hspace{5pt}& \boldNr&=(\boldWr^T\boldVr)^{-1}\boldWr^T\boldN\boldVr,\nonumber\\
\boldbr&=(\boldWr^T\boldVr)^{-1}\boldWr^T\boldb, &\boldcr&=\boldc\boldVr \label{volt_series_ROM}
\end{align}
 satisfies (\ref{volterra_series_interpolation}) and (\ref{volterra_series_interpolation_mu}) for each  
$\sigma_j$ and
 $\mu_j$, respectively, for $j=1, \dots, r$.
\end{thm}

\begin{proof}
 By Lemma \ref{lemma:sylvester_equation_solution}, we have shown how the columns of \boldVr, and \boldWr\ can be uniquely identified with the Volterra series we wish to match. Now define the skew projector $\mathbf{P}=\boldVr(\boldWr^T\boldVr)^{-1}\boldWr^T$.  Then
\begin{align}
\mathbf{P}(\boldVr\mathbf{\Lambda}-\boldA\boldVr-\boldN\boldVr\mathbf{U}-\boldb\bolde^T)=&\boldVr(\mathbf{\Lambda}-\boldAr-\boldNr\mathbf{U}-\boldbr\bolde^T)=\mathbf{0}.
\end{align}
Since $\boldVr$ is full rank, it follows that $\mathbf{\Gamma}=\boldI_r$ solves the projected Sylvester equation
\[\mathbf{\Gamma}\mathbf{\Lambda}-\boldAr\mathbf{\Gamma}-\boldNr\mathbf{\Gamma}\mathbf{U}^T=\boldbr\bolde^T.\]
By the same construction as above, the $jth$ column of $\mathbf{\Gamma}$, denoted 
by ${\boldsymbol\gamma}_j$, can be represented as 
\[\mathbf{\boldsymbol\gamma}_j=\sum \limits_{k=1}^{\infty} \sum\limits_{l_1}^r \sum \limits_{l_2}^r \cdots \sum \limits_{l_{k-1}}^r \eta_{ _{l_1,l_2, \dots, l_{k-1},j}}(\sigma_j\boldI_r-\boldAr)^{-1}\boldNr(\sigma_{l_{k-1}}\boldI_r-\boldAr)^{-1}\boldNr\cdots \boldNr(\sigma_{l_1}\boldI_r-\boldAr)^{-1}\boldbr.\]
Therefore
\begin{align}\label{volt_series_1}
\boldVr&\mathbf{\boldsymbol\gamma}_j=\mathbf{v}_j
\nonumber \\
& 
=\sum \limits_{k=1}^{\infty} \sum\limits_{l_1}^r 
 \cdots \sum \limits_{l_{k-1}}^r \eta_{ _{l_1, \dots, l_{k-1},j}}\boldVr(\sigma_j\boldI_r-\boldAr)^{-1}\boldNr(\sigma_{l_{k-1}}\boldI_r-\boldAr)^{-1}\boldNr\cdots \boldNr(\sigma_{l_1}\boldI_r-\boldAr)^{-1}\boldbr \nonumber \\
&=\sum \limits_{k=1}^{\infty} \sum\limits_{l_1}^r 
\cdots \sum \limits_{l_{k-1}}^r \eta_{ _{l_1,l_2, \dots, l_{k-1},j}}(\sigma_j\boldI-\boldA)^{-1}\boldN(\sigma_{l_{k-1}}\boldI-\boldA)^{-1}\boldN \cdots \boldN(\sigma_{l_1}\boldI-\boldA)^{-1}\boldb.
\end{align}
Multiplying equation (\ref{volt_series_1}) on the left by \boldc\ gives the desired result in terms of the interpolation conditions on $\sigma_j$.  For the interpolation conditions in the points $\mu_j$, observe that precisely the same construction of the columns of \boldWr\ follows from the proof given above applied to the equation
\begin{equation*}
\boldWr\mathbf{M}-\boldA^T\boldWr-\boldN^T\boldWr\mathbf{S}=\mathbf{c}^T\bolde
\end{equation*}
Now $\mathbf{P}^T=\boldWr(\boldVr^T\boldWr)^{-1}\boldV^T$ is a skew projection onto the range of \boldWr, and 
\begin{align*}
&\mathbf{P}^T(\boldWr\mathbf{M}-\boldA^T\boldWr-\boldN^T\boldWr\mathbf{S}-\boldc^T\bolde^T)\\
&=\boldWr(\boldVr^T\boldWr)^{-1}((\boldVr^T\boldWr)\mathbf{M}-\boldAr^T(\boldVr^T\boldWr)-\boldNr^T(\boldVr^T\boldWr)\mathbf{S}-\boldcr^T\bolde^T)\\
&=\mathbf{0}
\end{align*}
Since $\boldWr(\boldVr^T\boldWr)^{-1}$ is full rank, this implies that $\mathbf{\Xi}=\boldVr^T\boldWr \in \reals^{r\times r}$ solves
\begin{equation*}
\mathbf{\Xi}\mathbf{M}-\boldAr^T\mathbf{\Xi}-\boldNr^T\mathbf{\Xi}\mathbf{S}^T-\boldcr^T\bolde^T=\mathbf{0}
\end{equation*}
Again, by the construction given above, the columns $\boldsymbol\xi_j \in \reals^r$  of $\mathbf{\Xi}$ for $j=1,\dots, r$ can be represented as 
\[\mathbf{\boldsymbol\xi}_j=\sum \limits_{k=1}^{\infty} \sum\limits_{l_1}^r 
\cdots \sum \limits_{l_{k-1}}^r \widehat{\eta}_{ _{l_1, \dots, l_{k-1},j}}(\mu_j\boldI_r-\boldAr^T)^{-1}\boldNr^T(\mu_{l_{k-1}}\boldI_r-\boldAr^T)^{-1}\boldNr^T\cdots \boldNr^T(\mu_{l_1}\boldI_r-\boldAr^T)^{-1}\boldcr^T\]
And therefore 
\begin{align}\label{volt_series_2}
&\boldWr(\boldVr^T\boldWr)^{-1}\xi_j=\widetilde{\mathbf{w}}_j \nonumber\\
=&\sum \limits_{k=1}^{\infty} \sum\limits_{l_1}^r
 \cdots \sum \limits_{l_{k-1}}^r \widehat{\eta}_{ _{l_1, \dots, l_{k-1},j}}\boldWr(\boldVr^T\boldWr)^{-1}(\mu_j\boldI_r-\boldAr^T)^{-1}\boldNr^T\nonumber\\
&\hspace{4cm}\times (\mu_{l_{k-1}}\boldI_r-\boldAr^T)^{-1}\boldNr^T\cdots \boldNr^T(\mu_{l_1}\boldI_r-\boldAr^T)^{-1}\boldcr^T\nonumber\\
=&\sum \limits_{k=1}^{\infty} \sum\limits_{l_1}^r 
 \cdots \sum \limits_{l_{k-1}}^r \widehat{\eta}_{ _{l_1, \dots, l_{k-1},j}}(\mu_j\boldI-\boldA^T)^{-1}\boldN^T(\mu_{l_{k-1}}\boldI-\boldA^T)^{-1}\boldN^T \cdots \boldN^T(\mu_{l_1}\boldI-\boldA^T)^{-1}\boldc^T,
\end{align}
for $j=1, \dots, r$.  Taking the transpose of these equations and multiplying on the right by $\boldb$ yields the desired result for the interpolation points in $\mu_j$ and weights in $\mathbf{S}$.
\end{proof}

Theorem \ref{thm:volt_series_interp} shows how to construct a reduced bilinear system to solve the interpolation problem for the underlying Volterra series. Next, we connect this new interpolation framework to optimal approximation in the $\mathcal{H}_2$  norm.

\section{$\mathcal{H}_2$ optimal model reduction of bilinear systems} \label{sec:volterra_h2}
In this section we consider the $\mathcal{H}_2$ optimal model reduction problem and its solution.  Given an $n$ dimensional bilinear system $\zeta$, the $\mathcal{H}_2$ optimal model reduction problem for a given $r<n$ is to find the $r$-dimensional bilinear system $\widehat{\zeta}$ that satisifes

\begin{equation}
\widehat{\zeta}=\arg \min\limits_{\|\widetilde{\zeta}\|_{\mathcal{H}_2}<\infty } \|\zeta-\widetilde{\zeta}\|_{\mathcal{H}_2}
\end{equation}

Generalized Sylvester equation based first-order necessary conditions for $\mathcal{H}_2$ optimality were first derived by Lam and Zhang \cite{zhang2002h2}.  An alternative, but equivalent derivation was then given by Breiten and Benner in \cite{Breiten_H2}.  Their formulation of the necessary conditions for $\mathcal{H}_2$ optimality are obtained by taking derivatives of the $\mathcal{H}_2$ error expression with respect to the reduced order realization parameters.  Their results are summarized in the following theorem.

\begin{thm}[First-order necessary conditions for $\mathcal{H}_2$ optimality \cite{Breiten_H2}]\label{thm:bilinear_H2_opt_breiten}
Let the reduced bilinear model $\widetilde{\zeta}:=(\boldAr, \boldNr_1,\ldots,\boldNr_m, \boldBr, \boldCr)$ of dimension $r$ be a locally $\mathcal{H}_2$ optimal  approximation to the full-dimensional system ${\zeta}:=(\boldA, \boldN_1,\ldots,\boldN_m, \boldB, \boldC)$.  Let  $\mathbf{R} \widetilde{\mathbf{\Lambda}} \mathbf{R}^{-1}$ be the spectral decomposition of $\boldAr$, and define $\widehat{\boldB}=\boldBr^T\mathbf{R}^{-T}$, $\widehat{\boldC}=\boldCr\mathbf{R}$, $\widehat{\boldN}_k=\mathbf{R}^T(\boldNr)^T\mathbf{R}^{-T}$ for $k=1,\dots, m$.  
Moreover, let $\widetilde \bolde_i$ denote the $ith$ unit vector of length $r$, and $\bolde_i$ be the $ith$ unit vector whose length can be deduced from the context.
Then $\widetilde{\zeta}$ satisfies the following conditions:
For all $i=1,\dots,p$ and $j=1,\dots,r$,
\begin{align}
vec&(\boldI_p)^T(\bolde_i\widetilde{\bolde}_j^T\otimes \boldC)\bigg(-\widetilde{\mathbf{\Lambda}}\otimes \boldI_n-\boldI_r\otimes\boldA-\sum \limits_{k=1}^m \widehat{\boldN}_k^T\otimes\boldN_k\bigg)^{-1}(\widehat{\boldB}^T\otimes \boldB)vec(\boldI_m)\label{c_derivative_breiten}\\
&=vec(\boldI_p)^T(\bolde_i\widetilde\bolde_j^T\otimes \boldCr)\bigg(-\widetilde{\mathbf{\Lambda}}\otimes \boldI_r-\boldI_r\otimes\boldAr-\sum \limits_{k=1}^m (\widehat{\boldN}_k)^T\otimes\boldNr_k\bigg)^{-1}(\widehat{\boldB}^T\otimes \boldBr)vec(\boldI_m); \nonumber
\end{align}
\vspace{1ex}
for all $i=1,\dots,m$ and $j=1,\dots,r$,
\begin{align}
vec&(\boldI_p)^T(\widehat{\boldC}\otimes \boldC)\bigg(-\widetilde{\mathbf{\Lambda}}\otimes \boldI_n-\boldI_r\otimes\boldA-\sum \limits_{k=1}^m \widehat{\boldN}_k^T\otimes\boldN_k\bigg)^{-1}(\widetilde{\bolde}_j\bolde_i^T\otimes \boldB)vec(\boldI_m)\\
&=vec(\boldI_p)^T(\widehat{\boldC}\otimes \boldCr)\bigg(-\widetilde{\mathbf{\Lambda}}\otimes \boldI_r-\boldI_r\otimes\boldAr-\sum \limits_{k=1}^m (\widehat{\boldN}_k)^T\otimes\boldNr_k\bigg)^{-1}(\widetilde\bolde_j\bolde_i^T\otimes \boldBr)vec(\boldI_m);\nonumber
\end{align}
\vspace{1ex}
for all $i=1,\dots,r$,
\begin{align}
vec(\boldI_p)^T(\widehat{\boldC}\otimes \boldC)&\bigg(-\widetilde{\mathbf{\Lambda}}\otimes \boldI_n-\boldI_r\otimes\boldA-\sum \limits_{k=1}^m \widehat{\boldN}_k^T\otimes\boldN_k\bigg)^{-1}\nonumber\\
\times(\widetilde\bolde_i&\widetilde\bolde_i^T\otimes\boldI_n)\bigg(-\widetilde{\mathbf{\Lambda}}\otimes \boldI_n-\boldI_r\otimes\boldA-\sum \limits_{k=1}^m \widehat{\boldN}_k^T\otimes\boldN_k\bigg)^{-1}(\widehat{\boldB}^T\otimes \boldB)vec(\boldI_m)\nonumber\\
=vec(\boldI_p)^T(\widehat{\boldC}\otimes &\boldCr)\bigg(-\widetilde{\mathbf{\Lambda}}\otimes \boldI_r-\boldI_r\otimes\boldAr-\sum \limits_{k=1}^m (\widehat{\boldN}_k)^T\otimes\boldNr_k\bigg)^{-1}\label{H2_derivative_kronecker}\\
\times(\widetilde\bolde_i\widetilde\bolde_i^T &\otimes\boldI_r)\bigg(-\widetilde{\mathbf{\Lambda}}\otimes \boldI_r-\boldI_r\otimes\boldAr-\sum \limits_{k=1}^m (\widehat{\boldN}_k)^T\otimes\boldNr_k\bigg)^{-1}(\widehat{\boldB}^T\otimes \boldBr)vec(\boldI_m); 
\nonumber
\end{align} 
\vspace{1ex}
and, for $k=1,\ldots,m$, and for $i,j =1,\dots,r$,
\begin{align}
vec(\boldI_p)^T(\widehat{\boldC}\otimes \boldC)&\bigg(-\widetilde{\mathbf{\Lambda}}\otimes \boldI_n-\boldI_r\otimes\boldA-\sum \limits_{k=1}^m \widehat{\boldN}_k^T\otimes\boldN_k\bigg)^{-1}\nonumber\\
\times(\bolde_j&\bolde_i^T\otimes\boldN_k)\bigg(-\widetilde{\mathbf{\Lambda}}\otimes \boldI_n-\boldI_r\otimes\boldA-\sum \limits_{k=1}^m \widehat{\boldN}_k^T\otimes\boldN_k\bigg)^{-1}(\widehat{\boldB}^T\otimes \boldB)vec(\boldI_m)\nonumber\\
=vec(\boldI_p)^T(\widehat{\boldC}\otimes &\boldCr)\bigg(-\widetilde{\mathbf{\Lambda}}\otimes \boldI_r-\boldI_r\otimes\boldAr-\sum \limits_{k=1}^m \widehat{\boldN}_k^T\otimes\boldNr_k\bigg)^{-1}  \\
\times(\widetilde\bolde_j\widetilde\bolde_i^T&\otimes\boldNr_k)\bigg(-\widetilde{\mathbf{\Lambda}}\otimes \boldI_r-\boldI_r\otimes\boldAr-\sum \limits_{k=1}^m (\widehat{\boldN}_k)^T\otimes\boldNr_k\bigg)^{-1}(\widehat{\boldB}^T\otimes \boldBr)vec(\boldI_m).
\nonumber
\end{align}
\end{thm}
Based on Theorem \ref{thm:bilinear_H2_opt_breiten},
Benner and Breiten in \cite{Breiten_H2} have developed the Bilinear Iterative Rational Krylov Algorithm (B-IRKA); an iterative algorithm, which, upon convergence, produces a reduced bilinear system satisfying the first-order necessary conditions for $\mathcal{H}_2$ optimality given in Theorem \ref{thm:bilinear_H2_opt_breiten}. B-IRKA has successfully extended the Iterative Rational Krylov Algorithm (IRKA) of \cite{H2} for optimal-$\mathcal{H}_2$ approximation of linear systems to bilinear systems. 
B-IRKA has produced high-fidelity reduced models, outperformed  the balancing-based bilinear  reduction methods, and has become the method of choice in most cases; for details on B-IRKA, we refer the reader to the original source \cite{Breiten_H2}.  A brief sketch of B-IRKA is given below in Algorithm \ref{alg:BIRKA}.
\begin{figure}[h]
\begin{center}
    \framebox[5.125in][t]{
    \begin{minipage}[c]{5.0in}
    {

\begin{algorithm}[Bilinear Iterative Rational Krylov Algorithm (B-IRKA) \cite{Breiten_H2}]
\label{alg:BIRKA}
\\
\vspace{-0.7ex}

\textbf{Input}: \boldA, $\boldN_k$ for $k=1\ldots m$, \boldB, \boldC, \boldAr, $\boldNr_k$ for $k=1,\dots, m$, \boldBr, \boldCr\\
\textbf{Output}: $\boldAr^{\text{opt}}$, $\boldNr_k^{\text{opt}}$ for $k=1, \dots, m$, $\boldBr^{\text{opt}}$, $\boldCr^{\text{opt}}$
\begin{enumerate}
\item \textbf{While}: Change in $\widetilde{\mathbf{\Lambda}} >\epsilon$ do:
\item $\mathbf{R}\widetilde{\mathbf{\Lambda}} \mathbf{R}^{-1}=\boldAr$, $\widehat{\boldB}=\mathbf{R}^{-T}\boldBr$, $\widehat{\boldC}=\boldC \mathbf{R}$, $\widehat{\boldN}_k=\mathbf{R}^{-1}\boldNr_k\mathbf{R}$ for $k=1, \dots, m$.
\item Solve
\[\widetilde\boldV(-\widetilde{\mathbf{\Lambda}})-\boldA\widetilde\boldV-\sum \limits_{k=1}^m\boldN_k\widetilde\boldV\widehat{\boldN}_k^T=\boldB\widehat{\boldB}^T\] and 
\[\widetilde\boldW(-\widetilde{\mathbf{\Lambda}})-\boldA^T\widetilde\boldW-\sum \limits_{k=1}^m\boldN_k^T\widetilde\boldW\widehat{\boldN}_k=\boldC^T\widehat{\boldC}\]
\item $\widetilde\boldV=orth(\widetilde\boldV)$, $\widetilde\boldW=orth(\widetilde\boldW)$.
\item $\boldAr=(\widetilde\boldW^T\widetilde\boldV)^{-1}\widetilde\boldW^T\boldA\widetilde\boldV$, $\boldNr_k=(\widetilde\boldW^T\widetilde\boldV)^{-1}\widetilde\boldW^T\boldN_k\widetilde\boldV$ for $k=1, \dots, m$,\\ $\boldBr=(\widetilde\boldW^T\widetilde\boldV)^{-1}\widetilde\boldW^T\boldB$, $\widetilde{\boldC}=\boldC\widetilde\boldV$.
\item \textbf{end while}
\item $\boldAr^{\text{opt}}=\boldAr$, $\boldNr_k^{\text{opt}}=\boldNr_k$ for $k=1, \dots, m$, $\boldBr^{\text{opt}}=\boldBr$, $\boldCr^{\text{opt}}=\boldCr$
\end{enumerate}
\end{algorithm}

}
  \end{minipage}
    }
    \end{center}
  \end{figure}
  
 For $\boldN_k=\mathbf{0}$ for $k=1, \dots, m$, B-IRKA reduces to the Sylvester equation formulation of IRKA; see  \cite{benner_sylvester} for an effective implementation for the linear case using Sylvester equations.
The ability to satisfy the necessary conditions of Theorem \ref{thm:bilinear_H2_opt_breiten} by B-IRKA requires repeatedly solving the generalized Sylvester equations given in Step 3. of B-IRKA.  Unlike the linear case (i.e., when $\boldN_k=\mathbf{0}$ for $k=1, \dots, m$), solving these Sylvester equations is not always an easy task and requires solving a sequence of possibly dense linear systems of dimension $(nr)\times(nr)$, obtained by vectorizing the equations in Step 3. of B-IRKA. This means that as $r$ grows  moderately large,  say $r=30$, the computational cost per iteration of B-IRKA might become large as well. However, we want to emphasize that even with these numerical considerations, B-IRKA is the only optimal model reduction technique available for bilinear systems that is also applicable for large problems. In Section \ref{sec:TBIRKA}, we will propose a model reduction approach that performs comparably with the high quality of B-IRKA while only requiring solutions to the linear Sylvester equations, as in the case of IRKA.

\subsection{Multipoint interpolation and $\mathcal{H}_2$ optimality}
Breiten and Benner \cite{Breiten_H2} have observed that their necessary conditions
for $\mathcal{H}_2$ optimality
 are an algebraic analogue to the Sylvester equation formulation of rational interpolation conditions in the case of LTI systems \cite{Breiten_H2}.  We now present an analysis of  the necessary conditions of Theorem \ref{thm:bilinear_H2_opt_breiten} which makes an explicit connection to our multipoint Volterra series interpolation scheme.  Our analysis shows that the necessary conditions 
of Theorem \ref{thm:bilinear_H2_opt_breiten}
construed in terms of multipoint Volterra series interpolation yields rather satisfying generalizations of the interpolation-based necessary conditions originally introduced by Meier and Luenberger for $\mathcal{H}_2$ optimal approximation of LTI systems \cite{meieriii1967approximation}.

In order to obtain this result, we first prove the following lemma, which clarifies the relationship between the multi-point Volterra series interpolation conditions and the pole residue expansion of a SISO bilinear system.

\begin{lem}\label{lemma:clarified_pole_residue_expansion}
Let  $\zeta= (\boldA, \boldN, \boldb, \boldc)$ 
and  $\widetilde \zeta = (\boldAr, \boldNr, \boldbr, \boldcr)$ be SISO bilinear systems of dimension $n$ and $r$, respectively.
 Let  $\mathbf{R} \widetilde{\mathbf{\Lambda}} \mathbf{R}^{-1}$ be the spectral decomposition of $\boldAr$, and let $\widehat{\boldb}=\mathbf{R}^{-1}\boldbr$, $\widehat{\boldc}=\boldcr\,\mathbf{R}$, $\widehat{\boldN}=\mathbf{R}^{-1}\boldNr\mathbf{R}$.  Moreover, let the residues $\widetilde{\phi}_{l_1,\dots,l_k}$ for $k=1,\dots,\infty$ and $l_k=1,\dots,r$ of the transfer functions $\widetilde{H}_k(s_1,\dots,s_k)$ corresponding to the $kth$ order homogeneous subsystems of $\widetilde{\zeta}$ be defined as in Definition \ref{dfn:residues}.  Let \boldVr solve
\begin{align*}
\boldVr(-\widetilde{\mathbf{\Lambda}})-\boldA\boldVr-\boldN\boldVr\widehat{\boldN}^T=\boldb\widehat{\boldb}^T.\\
\end{align*}
Then
\begin{align}
&\widehat{\boldc}(\boldc\boldVr)^T
=\sum \limits_{k=1}^{\infty}\sum \limits_{l_1=1}^{r}\cdots\sum\limits_{l_{k}=1}^r \widetilde{\phi}_{ _{l_1,\dots,l_{k}}}H_k(-\widetilde{\lambda}_{l_1},\dots,-\widetilde{\lambda}_{k})\label{H2_sum_condition_lemma}
\end{align}
\end{lem}
\begin{proof}
Let $\boldU=\widehat{\boldN}$, $\mathbf{r}=\widehat{\boldb}$ and $\sigma_j=-\widetilde{\lambda}_j$ for $j=1,\dots, r$.  By applying the construction of the columns of \boldVr\ given in the proof of Theorem \ref{thm:volt_series_interp}, we have that
\begin{align}
\boldc\boldVr(:,j)&=\sum \limits_{k=1}^{\infty} \sum\limits_{l_1}^r \sum \limits_{l_2}^r \cdots \sum \limits_{l_{k-1}}^r \eta_{ _{l_1,l_2, \dots, l_{k-1},j}}\widehat{\boldb}_{l_1}H_k(-\widetilde{\lambda}_{l_1},-\widetilde{\lambda}_{l_2}, \dots, -\widetilde{\lambda}_j),\nonumber\\
\end{align}
where $\eta_{l_1,\dots,l_{k-1},j}=u_{j,l_{k-1}}u_{l_{k-1},l_{k-2}}\cdots u_{l_2,l_1} \text{ for } k\ge 2\text{ and }\eta_{l_1}=1 \text{ for }l_1=1, \dots, r$.
Now for each $j=1,\dots,r$, observe that by the definition of $\eta_{l_1,\dots,l_{k-1},j}$, for $k\ge2$
\begin{align}
\eta_{l_1,\dots,l_{k-1},j}\widehat{\boldb}_{l_1}=\widehat\boldN(j,l_{k-1})\widehat\boldN(l_{k-1},l_{k-2})\cdots\widehat\boldN(l_2,l_1)\widehat\boldb_{l_1}.
\end{align}
Therefore
\begin{align} 
&\sum \limits_{k=1}^{\infty} \sum\limits_{l_1}^r \sum \limits_{l_2}^r \cdots \sum \limits_{l_{k-1}}^r \eta_{ _{l_1,l_2, \dots, l_{k-1},j}}\widehat{\boldb}_{l_1}H_k(-\widetilde{\lambda}_{l_1},-\widetilde{\lambda}_{l_2}, \dots, -\widetilde{\lambda}_j)\nonumber\\
=&\sum \limits_{k=1}^{\infty} \sum\limits_{l_1}^r \sum \limits_{l_2}^r \cdots \sum \limits_{l_{k-1}}^r\widehat\boldN(j,l_{k-1})\widehat\boldN(l_{k-1},l_{k-2})\cdots\widehat\boldN(l_2,l_1)\widehat{\boldb}_{l_1}H_k(-\widetilde{\lambda}_{l_1},-\widetilde{\lambda}_{l_2}, \dots, -\widetilde{\lambda}_j)\nonumber\\
\end{align}
for $j=1,\dots,r$.
Hence,
\begin{align}
(\boldc\boldVr)^T=\begin{bmatrix}\sum \limits_{k=1}^{\infty} \sum\limits_{l_1}^r \sum \limits_{l_2}^r \cdots \sum \limits_{l_{k-1}}^r\widehat\boldN(1,l_{k-1})\widehat\boldN(l_{k-1},l_{k-2})\cdots\widehat\boldN(l_2,l_1)\widehat{\boldb}_{l_1}H_k(-\widetilde{\lambda}_{l_1},-\widetilde{\lambda}_{l_2}, \dots, -\widetilde{\lambda}_1)\\
\sum \limits_{k=1}^{\infty} \sum\limits_{l_1}^r \sum \limits_{l_2}^r \cdots \sum \limits_{l_{k-1}}^r\widehat\boldN(2,l_{k-1})\widehat\boldN(l_{k-1},l_{k-2})\cdots\widehat\boldN(l_2,l_1)\widehat{\boldb}_{l_1}H_k(-\widetilde{\lambda}_{l_1},-\widetilde{\lambda}_{l_2}, \dots, -\widetilde{\lambda}_2)\\
\vdots\\
\sum \limits_{k=1}^{\infty} \sum\limits_{l_1}^r \sum \limits_{l_2}^r \cdots \sum \limits_{l_{k-1}}^r\widehat\boldN(r,l_{k-1})\widehat\boldN(l_{k-1},l_{k-2})\cdots\widehat\boldN(l_2,l_1)\widehat{\boldb}_{l_1}H_k(-\widetilde{\lambda}_{l_1},-\widetilde{\lambda}_{l_2}, \dots, -\widetilde{\lambda}_r)
\end{bmatrix}. \label{eq:cV}
\end{align}
To complete the proof, we multiply the right-hand side of (\ref{eq:cV}) by $\widehat \boldc =\begin{bmatrix}\widehat\boldc_1&\widehat\boldc_2&\dots,\widehat\boldc_r \end{bmatrix}$ where 
$\widehat \boldc_i \in \reals$, for $i=1,\ldots,r$ denote the entries of $\widehat \boldc$.
Then, tracing the terms of a matrix vector products by their indices, together with the fact that the pole-residue decomposition of the $kth$ order transfer functions is unique yields
\begin{align}
&\begin{bmatrix}\widehat\boldc_1&\widehat\boldc_2&\dots,\widehat\boldc_r\end{bmatrix} (\boldc\boldVr)^T \nonumber \\
& = \begin{bmatrix}\widehat\boldc_1&\widehat\boldc_2&\dots,\widehat\boldc_r\end{bmatrix} \nonumber\\ 
&\qquad\times \begin{bmatrix}\sum \limits_{k=1}^{\infty} \sum\limits_{l_1}^r \sum \limits_{l_2}^r \cdots \sum \limits_{l_{k-1}}^r\widehat\boldN(1,l_{k-1})\widehat\boldN(l_{k-1},l_{k-2})\cdots\widehat\boldN(l_2,l_1)\widehat{\boldb}_{l_1}H_k(-\widetilde{\lambda}_{l_1},-\widetilde{\lambda}_{l_2}, \dots, -\widetilde{\lambda}_1)\\
\sum \limits_{k=1}^{\infty} \sum\limits_{l_1}^r \sum \limits_{l_2}^r \cdots \sum \limits_{l_{k-1}}^r\widehat\boldN(2,l_{k-1})\widehat\boldN(l_{k-1},l_{k-2})\cdots\widehat\boldN(l_2,l_1)\widehat{\boldb}_{l_1}H_k(-\widetilde{\lambda}_{l_1},-\widetilde{\lambda}_{l_2}, \dots, -\widetilde{\lambda}_2)\\
\vdots\\
\sum \limits_{k=1}^{\infty} \sum\limits_{l_1}^r \sum \limits_{l_2}^r \cdots \sum \limits_{l_{k-1}}^r\widehat\boldN(r,l_{k-1})\widehat\boldN(l_{k-1},l_{k-2})\cdots\widehat\boldN(l_2,l_1)\widehat{\boldb}_{l_1}H_k(-\widetilde{\lambda}_{l_1},-\widetilde{\lambda}_{l_2}, \dots, -\widetilde{\lambda}_r)\nonumber
\end{bmatrix}\\
=&\sum \limits_{k=1}^{\infty}\sum \limits_{l_1=1}^{r}\cdots\sum\limits_{l_{k}=1}^r \widetilde{\phi}_{ _{l_1,\dots,l_{k}}}H_k(-\widetilde{\lambda}_{l_1},\dots,-\widetilde{\lambda}_{k}), \label{eq:ctildecV}
\end{align}
concluding the proof.
\end{proof}

 Using Lemma \ref{lemma:clarified_pole_residue_expansion}, we now show that the $\mathcal{H}_2$ optimal necessary conditions  imply multipoint Volterra series interpolation conditions with weights given by the reduced order residues and interpolation points by the reflection of the poles of the reduced order transfer functions across the imaginary axis.

\begin{thm}\label{thm:H2_sum_condition}
Let $\zeta$ be a SISO system of dimension $n$, and let $\widetilde{\zeta}=(\boldAr,\boldNr,\boldbr,\boldcr)$ be an $\mathcal{H}_2$ optimal approximation of $\zeta$ of dimension $r$.  
Then $\widetilde{\zeta}$ satisfies the following multipoint Volterra series interpolation conditions.
\begin{align}\label{H2_sum_condition}
\sum \limits_{k=1}^{\infty}\sum \limits_{l_1=1}^{r}\cdots\sum\limits_{l_{k}=1}^r \widetilde{\phi}_{ _{l_1,\dots,l_{k}}}H_k(-\widetilde{\lambda}_{l_1},\dots,-\widetilde{\lambda}_{k})\nonumber\\
=\sum \limits_{k=1}^{\infty}\sum \limits_{l_1=1}^{r}\cdots\sum\limits_{l_{k-1}=1}^r \widetilde{\phi}_{ _{l_1,\dots,l_{k}}}\widetilde{H}_k(-\widetilde{\lambda}_{l_1},\dots,-\widetilde{\lambda}_k ),
\end{align}
and
\begin{align}\label{H2_sum_partials_condition}
\sum \limits_{k=1}^{\infty}\sum \limits_{l_1=1}^{r}\cdots\sum\limits_{l_{k}=1}^r \widetilde{\phi}_{ _{l_1,\dots,l_{k}}}\Big( \sum \limits_{j=1}^{k} \frac{\partial}{\partial s_j}H_k(-\widetilde{\lambda}_{l_1},\dots,-\widetilde{\lambda}_{k})\Big)\nonumber\\
=\sum \limits_{k=1}^{\infty}\sum \limits_{l_1=1}^{r}\cdots\sum\limits_{l_k=1}^r \widetilde{\phi}_{ _{l_1,\dots,l_k}} \Big(\sum \limits_{j=1}^{k}  \frac{\partial}{\partial s_j}\widetilde{H}_k(-\widetilde{\lambda}_{l_1},\dots,-\widetilde{\lambda}_{l_{k}} )\Big)
\end{align}
where $\widetilde{\phi}_{ _{l_1,\dots,l_k}}$, and $\widetilde{\lambda}_{l_i}$ are, respectively, the residues and poles of the transfer functions $\widetilde{H}_k$ associated with $\widetilde{\zeta}$,
and $\frac{\partial}{\partial s_j} {H}_k(-{\lambda}_{l_1},\dots,-{\lambda}_{l_{k}} )$ denotes the the partial derivative of  ${H}_k(s_1,\dots,s_k)$ with respect to $s_j$ evaluated at $(s_1,\ldots,s_k) = (-\lambda_{l_1},\ldots,-\lambda_{l_k})$.
\end{thm}

\begin{proof}
Let  $\mathbf{R} \widetilde{\mathbf{\Lambda}} \mathbf{R}^{-1}$ be the spectral decomposition of $\boldAr$, and let $\widehat{\boldb}=\mathbf{R}^{-1}\boldbr$, $\widehat{\boldc}=\boldc \mathbf{R}$, $\widehat{\boldN}=\mathbf{R}^{-1}\boldNr\mathbf{R}$.  Moreover, let $\boldV$ and \boldW\ solve
\begin{align}
\boldV(-\widetilde{\mathbf{\Lambda}})-\boldA\boldV-\boldN\boldV\widehat{\boldN}^T=\boldb\widehat{\boldb}^T \label{equation1}\\
\boldW(-\mathbf{\Lambda})-\boldA^T\boldW-\boldN^T\boldW\widehat{\boldN}^T=\boldc^T\widehat{\boldc} \label{equation2}
\end{align}

By applying the $vec$ operator to equations (\ref{equation1}) and (\ref{equation2}), we have that \[vec(\boldV)=\bigg(-\widetilde{\mathbf{\Lambda}}\otimes\boldI_n-\boldI_r\otimes\boldA-\widehat{\boldN}^T\otimes\boldN\bigg)^{-1}(\widehat{\boldb}^T\otimes\boldb).\]  Thus, 
\begin{align*}
(\bolde_j^T\otimes\boldc) vec(\boldV)=\boldc\boldV(:,j)
\end{align*}
 is equivalent to the left-hand side of necessary condition (\ref{c_derivative_breiten}).  Applying Lemma \ref{lemma:clarified_pole_residue_expansion} to both sides of (\ref{c_derivative_breiten}) gives
\begin{align*}
\sum \limits_{k=1}^{\infty}\sum \limits_{l_1=1}^{r}\cdots&\sum\limits_{l_{k}=1}^r \widetilde{\phi}_{ _{l_1,\dots,l_{k}}}H_k(-\widetilde{\lambda}_{l_1},\dots,-\widetilde{\lambda}_{k})\nonumber\\
&=\sum \limits_{k=1}^{\infty}\sum \limits_{l_1=1}^{r}\cdots\sum\limits_{l_{k-1}=1}^r \widetilde{\phi}_{ _{l_1,\dots,l_{k}}}\widetilde{H}_k(-\widetilde{\lambda}_{l_1},\dots,-\widetilde{\lambda}_k ).
\end{align*}

The second equality (\ref{H2_sum_partials_condition}) follows from  condition (\ref{H2_derivative_kronecker}).  Simple algebra shows that the right-hand-side of equality (\ref{H2_derivative_kronecker}) is equivalent to the product  $\boldW(:,j)^T\boldV(:,j)$.  This is equivalent to 
{\small
\begin{align}\label{derivative_product}
&\boldW(:,j)^T\boldV(:,j) = \nonumber \\
& \Big(\sum \limits_{k=1}^{\infty}\sum \limits_{l_1=1}^r\cdots \sum\limits_{l_{k-1}=1}^{r}\widehat\boldc_{l_1}\eta_{ _{j,l_{k-1},\ldots, l_1}}\boldc(-\widetilde{\lambda}_{l_1}\boldI_n-\boldA)\boldN\cdots\boldN(-\widetilde{\lambda}_{l_{k-1}}\boldI_n-\boldA)^{-1}\boldN(-\widetilde{\lambda}_j\boldI_n-\boldA)^{-1}\Big)  \nonumber \\
& \times\Big(\sum_{k=1}^{\infty}\sum\limits_{r_1=1}^r\cdots\sum\limits_{r_{k-1}=1}^r\widehat{\boldb}_{r_1}\eta_{r_1,\ldots,r_{k-1},j}(-\widetilde{\lambda}_j\boldI_n-\boldA)^{-1}\boldN(-\widetilde{\lambda}_{r_{k-1}}\boldI_n-\boldA)^{-1}\boldN\cdots\boldN(-\widetilde{\lambda}_{r_1}\boldI_n-\boldA)^{-1}\boldb\Big). \nonumber
\end{align}
}
Expanding over the first few terms in $k$ is sufficient to establish the general pattern:
\begin{align}
\boldWr(:,j)&^T\boldVr(:, j)= \nonumber \\
& \widehat\boldc_{j}\widehat\boldb_{j}\boldc(-\widetilde{\lambda}_j\boldI_n-\boldA)^{-2}\boldb+\sum \limits_{r_1=1}^r\widehat\boldc_j\eta_{r_1,j}\widehat\boldb_{r_1}\Big(\boldc(-\widetilde{\lambda}_j\boldI_n-\boldA)^{-2}\boldN(-\widetilde{\lambda}_{r_1}\boldI_n-\boldA)^{-1}\boldb \nonumber \\
&+\sum\limits_{l_1=1}^r\widehat\boldc_{l_1}\eta_{j,l_1}\widehat\boldb_j\boldc(-\widetilde{\lambda}_{l_1}\boldI_n-\boldA)^{-1}\boldN(-\widetilde{\lambda}_{j}\boldI_n-\boldA)^{-2}\boldb\Big) +\nonumber\\
&+\sum \limits_{r_1=1}^{r}\sum \limits_{r_2=1}^r\widehat\boldc_{j}\eta_{ _{r_1,r_2,j}}\widehat\boldb_{r_1}(\boldc(-\widetilde{\lambda}_j\boldI_n-\boldA)^{-2}\boldN(-\widetilde{\lambda}_{r_2}\boldI_n-\boldA)^{-1}\boldN(-\widetilde{\lambda}_{r_1}\boldI_n-\boldA)^{-1}\boldb \nonumber\\
&+\sum \limits_{l_1=1}^{r}\sum \limits_{l_2=1}^r\widehat\boldc_{l_1}\eta_{ _{j,l_2,l_1}}\widehat\boldb_{j}(\boldc(-\widetilde{\lambda}_{l_1}\boldI_n-\boldA)^{-1}\boldN(-\widetilde{\lambda}_{l_2}\boldI_n-\boldA)^{-1}\boldN(-\widetilde{\lambda}_{j}\boldI_n-\boldA)^{-2}\boldb\nonumber\\
&+\sum \limits_{l_1=1}^{r}\sum \limits_{r_1=1}^r\widehat\boldc_{l_1}\eta_{j,l_1}\eta_{r_1,j}\widehat\boldb_{r_1}(\boldc(-\widetilde{\lambda}_{l_1}\boldI_n-\boldA)^{-1}\boldN(-\widetilde{\lambda}_j\boldI_n-\boldA)^{-2}\boldN(-\widetilde{\lambda}_{r_1}\boldI_n-\boldA)^{-1}\boldb \nonumber\\ \nonumber
&+ \dots ,
\end{align}

where the weights $\eta_{ _{r_1,r_2,j}}, \eta_{ _{j,l_2,l_1}}$ etc. are defined in (\ref{weights}), and the indices in $r_j$ and $l_j$ keep track of the cases where terms on the right are multiplied by terms on the left and vice versa in the obvious way.    The expansion of the product for the solution of the reduced order matrices follows similarly.  Thus, $\sum \limits_{j=1}^r \boldW(:,j)^T\boldV(:,j)$ gives the desired expression for the derivatives as:
\begin{align*}
\sum \limits_{k=1}^{\infty}\sum \limits_{l_1=1}^{r}\cdots\sum\limits_{l_{k}=1}^r \widetilde{\phi}_{ _{l_1,\dots,l_{k}}}\Big( \sum \limits_{j=1}^{k} \frac{\partial}{\partial s_j}H_k(-\widetilde{\lambda}_{l_1},\dots,-\widetilde{\lambda}_{k})\Big)\nonumber\\
=\sum \limits_{k=1}^{\infty}\sum \limits_{l_1=1}^{r}\cdots\sum\limits_{l_k=1}^r \widetilde{\phi}_{ _{l_1,\dots,l_k}} \Big(\sum \limits_{j=1}^{k}  \frac{\partial}{\partial s_j}\widetilde{H}_k(-\widetilde{\lambda}_{l_1},\dots,-\widetilde{\lambda}_{l_{k}} )\Big).
\end{align*} Since all the terms $j$ are equal on both sides of equation (\ref{H2_derivative_kronecker}), the second result follows.
\end{proof}

In the SISO case, the generalized Sylvester equations in Step 3. of BIRKA are a special case of the multipoint interpolation conditions presented in Theorem \ref{thm:volt_series_interp} where the interpolation points are  $-\lambda(\boldAr)$, and the weights generated by $\widehat{\boldN}_r$ are simply the residues of the $kth$ order transfer functions. Upon convergence of BIRKA, the resulting reduced order system is an $\mathcal{H}_2$ approximation satisfying the first-order necessary conditions. Hence, all of the convergence criteria associated with the corresponding Volterra series interpolation expressions are satisfied, and the fixed point of the BIRKA iteration satisfies Theorem 4.2.  As we have noted, enforcing these multipoint interpolation conditions requires exactly solving the generalized Sylvester equations in step 3.) of B-IRKA.  The proof of Theorem \ref{thm:volt_series_interp} shows that the interpolation data can be constructed by iteratively solving and then summing the solutions of ordinary Sylvester equations. This suggests the possibility of enforcing the multipoint interpolation conditions on partial sums, making it possible to exploit the ordinarily fast decay in the terms of the Volterra series expansion of \Sys.  In what follows we show that this corresponds to solving the $\mathcal{H}_2$ optimal approximation for polynomial systems given by truncating the Volterra series expressions for the external representation of the bilinear system to the first $N$ terms in the series.

\section{A truncated $\mathcal{H}_2$ optimal  model reduction algorithm}
\label{sec:TBIRKA}
In this section, after introducing $\mathcal{H}_2$ optimality conditions for polynomial systems, we introduce an
effective numerical  algorithm for model reduction of MIMO bilinear systems.

\subsection{Polynomial bilinear systems}

Let us first consider polynomial systems generated by truncating the Volterra series of a bilinear system.

\begin{dfn}
Given a MIMO bilinear system $\zeta$ with realization (\boldA, $\boldN_1,\dots, \boldN_m$ \boldB, \boldC), define the polynomial system $\zeta^N:\mathbf{u} \in \mathcal{U} \rightarrow \reals^{p}$ to be the operator mapping inputs $\mathbf{u}(t)$ to outputs $\mathbf{y}(t)$ defined by the relation
\begin{align*}
\mathbf{y}(t)=&\sum_{k=1}^{N}\int_0^{\infty}\cdots\int_0^{\infty}\mathbf{h}_k(t_1,\dots,t_k)\mathbf{u}(t-\sum_{i=1}^k t_i)\otimes\mathbf{u}(t-\sum_{i=2}^k t_i)\otimes \cdots \otimes\mathbf{u}(t-t_k)\mathrm d t_k \cdots \mathrm d t_1
\end{align*}
where $\mathbf{h}_k$ is given by equation (\ref{mimo_kernel}).
Note that a polynomial system can also be identified with its sequence of transfer functions $(\mathbf{H}_1(s_1), \mathbf{H}_2(s_1,s_2), \dots, \mathbf{H}_N(s_1, \dots,s_N))$
where $\mathbf{H}_k(s_1, \dots, s_k)$ is given by equation (\ref{MIMO_TF}).
\end{dfn}
Trivially, every polynomial system $\zeta^N$ has a finite $\mathcal{H}_2$ norm, and due to Plancherel's equality
\begin{equation}
\|\zeta^N\|_{\mathcal{H}_2}=\sqrt{\sum \limits_{k=1}^{N} \int\limits_{0}^{\infty} \cdots \int \limits_{0}^{\infty} \|\mathbf{h}_k(t_1, \dots, t_k)\|_F^2\mathrm{d}t_k \cdots \mathrm{d}t_1}\label{polynomial_H2_norm}
\end{equation}
Thus, the operators $\zeta^N$ converge strongly to $\zeta$.  It follows that if $\{\widetilde\zeta^{N}\}$ is a sequence of $r-dimensional$ locally optimal polynomial approximations to $\zeta^{N}$ that converges in norm to the system $\widetilde{\zeta}$, then $\widetilde{\zeta}$ is a locally optimal approximation to $\zeta$.  We will therefore derive necessary conditions for $\mathcal{H}_2$ optimality of an $r$-dimensional polynomial approximation $\widetilde\zeta^{N}$  of an $n$ dimensional polynomial system $\zeta^{N}$. First, 
we need an expression for the $\mathcal{H}_2$ error norm  $\|\zeta^N-\widetilde\zeta^N\|_{\mathcal{H}_2}$.
\begin{lemma}
Let $\widetilde{\zeta}^N$ be a polynomial system generated by truncating the bilinear system 
$\widetilde{\zeta}=(\boldAr, \boldNr_1,\dots, \boldNr_m, \boldBr, \boldCr)$ of dimension $r$.  Let $\widetilde{\mathbf{\Lambda}}=\mathbf{T}^{-1}\boldAr\mathbf{T}$ be the spectral decomposition of \boldAr, and define $\widehat{\boldB}=\mathbf{T}^{-1}\boldBr$, $\widehat{\boldC}=\boldCr\mathbf{T}$ and  $\widehat{\boldN}_j=\mathbf{T}^{-1}\boldNr_j\mathbf{T}$ for $j=1,\dots, m$. Then,
\begin{align}
E_N&=\|\zeta^N-\widetilde\zeta^N\|_{\mathcal{H}_2}^2\nonumber\\
=&vec(\boldI_p)\Big([\boldC -\widehat{\boldC}]\otimes[\boldC -\widehat{\boldC}]\Big)  \sum \limits_{k=0}^N\Bigg[\Bigg(-\begin{bmatrix}\boldA& \mathbf{0}\\ \mathbf{0} & \widetilde{\mathbf{\Lambda}}\end{bmatrix}\otimes\begin{bmatrix}\boldI_n& \mathbf{0}\\ \mathbf{0} & \boldI_r\end{bmatrix}-\begin{bmatrix}\boldI_n& \mathbf{0}\\ \mathbf{0} & \boldI_r\end{bmatrix}\otimes\begin{bmatrix}\boldA& \mathbf{0}\\ \mathbf{0} & \widetilde{\mathbf{\Lambda}}\end{bmatrix}\Bigg)^{-1}\nonumber\\
&\times\sum\limits_{j=1}^m\begin{bmatrix}\boldN_j& \mathbf{0}\\ \mathbf{0}&\boldNr_j\end{bmatrix}\otimes\begin{bmatrix}\boldN_j& \mathbf{0}\\ \mathbf{0}&\boldNr_j\end{bmatrix}\Bigg]^k\Bigg(-\begin{bmatrix}\boldA& \mathbf{0}\\ \mathbf{0} & \widetilde{\mathbf{\Lambda}}\end{bmatrix}\otimes\begin{bmatrix}\boldI_n& \mathbf{0}\\ \mathbf{0} & \boldI_r\end{bmatrix}-\begin{bmatrix}\boldI_n& \mathbf{0}\\ \mathbf{0} & \boldI_r\end{bmatrix}\otimes\begin{bmatrix}\boldA& \mathbf{0}\\ \mathbf{0} & \widetilde{\mathbf{\Lambda}}\end{bmatrix}\Bigg)^{-1}\nonumber\\
&\times \begin{bmatrix}\boldB\\ \widehat{\boldB}\end{bmatrix} \otimes \begin{bmatrix}\boldB\\ \widehat{\boldB}\end{bmatrix}vec(\boldI_m). \label{EN}
\end{align}
\end{lemma}
\begin{proof}
 Using the equivalent form of the $\mathcal{H}_2$ norm in (\ref{norm_equivalence}) for the case of polynomial systems 
 in (\ref{polynomial_H2_norm}), we obtain
\begin{align}\label{trunc_theorem_gram_sum}
\big\|\zeta^N-\widetilde\zeta^N\big\|_{\mathcal{H}_2}^2&=\begin{bmatrix}\boldC &-\widehat{\boldC}\end{bmatrix}\bigg(\sum \limits_{k=1}^N \mathbf{P}_k\bigg)\begin{bmatrix}\boldC\\ -\widehat{\boldC}\end{bmatrix},
\end{align}
where $\mathbf{P}_1$ solves
\[-\begin{bmatrix}\boldA& \mathbf{0}\\ \mathbf{0} & \widetilde{\mathbf{\Lambda}}\end{bmatrix}\mathbf{P}_1-\begin{bmatrix}\boldA^T& \mathbf{0}\\ \mathbf{0} & -\widetilde{\mathbf{\Lambda}}^T\end{bmatrix}\mathbf{P}_1=\begin{bmatrix}\boldB\\ \widehat{\boldB}\end{bmatrix}\begin{bmatrix}\boldB^T & \widehat{\boldB}^T\end{bmatrix}\]
and for $k=2,\ldots,N$, $\mathbf{P}_k$ solves
\begin{align*}
-\begin{bmatrix}\boldA& \mathbf{0}\\ \mathbf{0} & \widetilde{\mathbf{\Lambda}}\end{bmatrix}\mathbf{P}_k-\mathbf{P}_k\begin{bmatrix}\boldA^T& \mathbf{0}\\ \mathbf{0} & \widetilde{\mathbf{\Lambda}}^T\end{bmatrix}=\sum \limits_{j=1}^m\begin{bmatrix}\boldN_j& \mathbf{0}\\ \mathbf{0}&\widehat{\boldN}_j\end{bmatrix}\mathbf{P}_{k-1}\begin{bmatrix}\boldN_j^T& \mathbf{0}\\ \mathbf{0}&\widehat{\boldN}_j^T\end{bmatrix}.
\end{align*}
Applying the $vec$ operator to the Lyapunov equation for $\mathbf{P}_1$ gives 
\begin{equation}\label{P_0_expression}
vec(\mathbf{P}_1)=\Bigg(-\begin{bmatrix}\boldA& \mathbf{0}\\ \mathbf{0} & \widetilde{\mathbf{\Lambda}}\end{bmatrix}\otimes\begin{bmatrix}\boldI_n& \mathbf{0}\\ \mathbf{0} & \boldI_r\end{bmatrix}-\begin{bmatrix}\boldI_n& \mathbf{0}\\ \mathbf{0} & \boldI_r\end{bmatrix}\otimes\begin{bmatrix}\boldA& \mathbf{0}\\ \mathbf{0} & \widetilde{\mathbf{\Lambda}}\end{bmatrix}\Bigg)^{-1}\begin{bmatrix}\boldB\\ \widehat{\boldB}\end{bmatrix} \otimes\begin{bmatrix}\boldB\\ \widehat{\boldB}\end{bmatrix}
\end{equation}
and for $\mathbf{P}_k$ gives
\begin{align}
vec(\mathbf{P}_k)=\Bigg(-\begin{bmatrix}\boldA& \mathbf{0}\\ \mathbf{0} & \widetilde{\mathbf{\Lambda}}\end{bmatrix}
\otimes&\begin{bmatrix}\boldI_n& \mathbf{0}\\ \mathbf{0} & \boldI_r\end{bmatrix}-\begin{bmatrix}\boldI_n& \mathbf{0}\\ \mathbf{0} & \boldI_r\end{bmatrix}\otimes\begin{bmatrix}\boldA& \mathbf{0}\\ \mathbf{0} & \widetilde{\mathbf{\Lambda}}\end{bmatrix}\Bigg)^{-1} \nonumber \\
\label{P_K_expressions}
&\times
\sum \limits_{j=1}^m\begin{bmatrix}\boldN_j& \mathbf{0}\\ \mathbf{0}&\widehat{\boldN}_j\end{bmatrix}\otimes\begin{bmatrix}\boldN_j& \mathbf{0}\\ \mathbf{0}&\widehat{\boldN}_j\end{bmatrix}vec(\mathbf{P}_{k-1}).
\end{align}

Applying the vec operator to the sum (\ref{trunc_theorem_gram_sum}) and successively substituting the expressions (\ref{P_0_expression}) and (\ref{P_K_expressions}) into the sum gives the desired result \ref{EN}.
\end{proof}

\subsection{$\mathcal{H}_2$ optimality for polynomial systems}
Now that we have an explicit expression for the error  $E_N$, we can differentiate this expression with respect to the reduced model quantities $\widehat{\boldA},\widehat{\boldN}_k,\widehat{\boldB_k}$ and $\widehat{\boldC}$ to obtain the necessary conditions for optimality.
This differentiation procedure will be greatly simplified by using the following result from \cite{Breiten_H2}.
\begin{lemma}[\cite{Breiten_H2}]\label{lemma:product_rule}
Let $\boldC(x)\in \reals^{p\times n}$, $\boldA(y), \mathbf{G}_k \in \reals^{n\times n}$, and $\mathbf{K}\in \reals^{n\times m}$ with 
\begin{equation*}
\mathcal{L}(y)=-\boldA(y)\otimes\boldI-\boldI\otimes\boldA(y)-\sum\limits_{k=1}^m\mathbf{G}_k\otimes\mathbf{G}_k
\end{equation*}
and assume that \boldC\ and \boldA\ are differentiable with respect to $x$, and $y$.  Then
\begin{align*}
&\frac{\partial}{\partial x}[(\text{vec}(\boldI_p))^T(\boldC(x)\otimes\boldC(x))\mathcal{L}(y)^{-1}(\mathbf{K}\otimes\mathbf{K})\text{vec}(\boldI_m)\\
&=2(\text{vec}(\boldI_p))^T(\frac{\partial}{\partial x}\boldC(x)\otimes\boldC(x))\mathcal{L}(y)^{-1}(\mathbf{K}\otimes\mathbf{K})vec(\boldI_m)
\end{align*}
and
\begin{align*}
&\frac{\partial}{\partial y}[(\text{vec}(\boldI_p))^T(\boldC(x)\otimes\boldC(x))\mathcal{L}(y)^{-1}(\boldB\otimes\boldB)vec(\boldI_m)]\\
&=2(vec(\boldI_p))^T(\boldC(x)\otimes\boldC(x))\mathcal{L}(y)^{-1}(\frac{\partial}{\partial y}\boldA(y)\otimes \boldI)\mathcal{L}^{-1}(y)(\mathbf{K}\otimes\mathbf{K})vec(\boldI_m).
\end{align*}

\end{lemma}

Another important tool for analyzing the resulting expressions for the derivative of $E_N$ is the permutation matrix

\begin{equation}\label{permutation_matrix}
\mathbf{M}=\begin{bmatrix}\boldI_r\otimes\begin{bmatrix}\boldI_n\\\mathbf{0}\end{bmatrix}&\boldI_r\otimes\begin{bmatrix}\mathbf{0}^T\\\mathbf{I}_r\end{bmatrix}\end{bmatrix}
\end{equation}
 introduced in \cite{Breiten_H2}.  
Given matrices $\mathbf{H}, \mathbf{K}\in \reals^{r\times r}$ and $\mathbf{L}\in \reals^{n\times n}$, the permutation $\mathbf{M}$ satisfies
\begin{align*}
&\mathbf{M}^T\left(\mathbf{H}^T\otimes\begin{bmatrix}\mathbf{L}&\mathbf{0}\\ \mathbf{0}&\mathbf{K}\end{bmatrix}\right) \mathbf{M}\\
&=\begin{bmatrix} \boldI_r\otimes \begin{bmatrix}\boldI_n \mathbf{0}^T\end{bmatrix}& \boldI_r\otimes \begin{bmatrix}\mathbf{0}&\boldI_r\end{bmatrix}\end{bmatrix}\left(\mathbf{H}^T\otimes\begin{bmatrix}\mathbf{L}&\mathbf{0}\\ \mathbf{0}&\mathbf{K}\end{bmatrix}\right)\begin{bmatrix}\boldI_r\otimes\begin{bmatrix}\boldI_n\\\mathbf{0}\end{bmatrix}&\boldI_r\otimes\begin{bmatrix}\mathbf{0}^T\\\mathbf{I}_r\end{bmatrix}\end{bmatrix}\\
&=\begin{bmatrix} \boldI_r\otimes \begin{bmatrix}\boldI_n \mathbf{0}^T\end{bmatrix}& \boldI_r\otimes \begin{bmatrix}\mathbf{0}&\boldI_r\end{bmatrix}\end{bmatrix}\begin{bmatrix}\mathbf{H}^T\otimes\begin{bmatrix}\mathbf{L}\\\mathbf{0}\end{bmatrix}& \mathbf{H}^T\otimes \begin{bmatrix}\mathbf{0}^T\\ \mathbf{K}\end{bmatrix}\end{bmatrix}\\
&=\begin{bmatrix}\mathbf{H}^T\otimes \mathbf{L}&\mathbf{0}\\ \mathbf{0} &\mathbf{H}^T\otimes \mathbf{K}\end{bmatrix}
\end{align*}
Finally, while the analysis of the cost function $E_N$ is most easily done in the Kronecker product formulation, we will retranslate the resulting necessary conditions into their Sylvester equation formulation to shorten the presentation in their later use.  To that end, we will need the solutions $\boldV_1$, $\boldW_1$ of the ordinary Sylvester equations
\begin{align}
\boldV_1(-\widetilde{\mathbf{\Lambda}})-\boldA\boldV_1=\boldB\widehat{\boldB}^T\\
\boldW_1(-\widetilde{\mathbf{\Lambda}})-\boldA^T\boldW_1=\boldC^T\widehat{\boldB}
\end{align}
and for $k>1$ the solutions $\boldV_k$, $\boldW_k$ of the ordinary Sylvester equations
\begin{align}
\boldV_k(-\widetilde{\mathbf{\Lambda}})-\boldA\boldV_k=\sum \limits_{j=1}^m\boldN_j \boldV_{k-1} \widehat{\boldN}_j^T
\boldW_k(-\widetilde{\mathbf{\Lambda}})-\boldA^T\boldW_k=\sum \limits_{j=1}^m \boldN_j^T \boldW_{k-1}\widehat{\boldN}_j
\end{align}
Furthermore, define the matrices 
\begin{equation}
\mathbf{S}_N=\sum \limits_{k=1}^N \boldV_k \text{ and } \mathbf{U}_N=\sum \limits_{k=1} \boldW_k
\end{equation}
Let $\boldVr_k$ and $\boldWr_k$, and $\widetilde{\mathbf{S}}_N$ and $\widetilde{\mathbf{U}}_N$ denote the solutions of the above equations where the reduced dimension parameters \boldAr, \boldBr, \boldCr, $\boldNr_j$ replace \boldA, \boldB, \boldC, $\boldN_j$ in all the appropriate places.
\begin{theorem} \label{h2conditions_truncated}
Let $\widetilde{\zeta}^N$ be a polynomial system generated by truncating the bilinear system 
$\widetilde{\zeta}=(\boldAr, \boldNr_1,\dots, \boldNr_m, \boldBr, \boldCr)$ of dimension $r$.  Let $\widetilde{\mathbf{\Lambda}}=\mathbf{T}^{-1}\boldAr\mathbf{T}$ be the spectral decomposition of \boldAr, and define $\widehat{\boldB}=\mathbf{T}^{-1}\boldBr$, $\widehat{\boldC}=\boldCr\mathbf{T}$ and  $\widehat{\boldN}_j=\mathbf{T}^{-1}\boldNr_j\mathbf{T}$ for $j=1,\dots, m$. If $\widetilde{\zeta}^N$ is a locally optimal $\mathcal{H}_2$ approximation to $\zeta^N$, then,
\begin{equation}
\text{trace}(\boldC\mathbf{S}_N \bolde_i\bolde_j^T)=\text{trace}(\boldCr \widetilde{\mathbf{S}}_N\bolde_i\bolde_j^T)\text{ } i=1,\dots,r\text{ } j=1,\dots, p \label{Sylvester_necessary_C}
\end{equation}
\begin{equation}
(\mathbf{U}_N(:,i))^T\mathbf{S}_N(:,i)=(\widetilde{\mathbf{U}}_N(:,i))^T\widetilde{\mathbf{S}}_N(:,i), \text{ } i=1,\dots, r \label{lambda_Sylvester_condition}
\end{equation}
\begin{equation}
\text{trace}(\boldB^T\mathbf{U}_N\bolde_i \bolde_j^T)=\text{trace}(\boldBr^T\widetilde{\mathbf{U}}_N\bolde_i\bolde_j^T), \text{ }  i=1,\dots,r\text{ } j=1,\dots, m \label{truncated_Sylvester_b_necessary}
\end{equation}
\begin{equation}
(\mathbf{U}_N(:,i))^T\boldN_k\mathbf{S}_N(:,j)=(\widetilde{\mathbf{U}}_N(:,i))^T\boldNr_k\widetilde{\mathbf{S}}_N(:,j), \text{ } i,j = 1,\dots, r\text{ } k = 1,\dots, m\label{truncated_n_Sylvester_necessary}
\end{equation}
\end{theorem}

\begin{proof}
The proof is given in Section \ref{appendix_proof}.
\end{proof}
\begin{rem}
Taking $N \rightarrow \infty$ yields the necessary conditions of Theorem \ref{thm:bilinear_H2_opt_breiten}.  This follows from the fact that at a local minimum, the solution $\boldV$ of 
\begin{equation*}
\boldV(-\widetilde{\mathbf{\Lambda}})-\boldA\boldV-\sum \limits_{j=1}^m\boldN_j\boldV\widehat{\boldN}_j^T=\boldB\widehat{\boldB}^T
\end{equation*}
in Step 3.) of B-IRKA is given as the series $\sum \limits_{k=0}^{\infty}\boldV_k$ where
\begin{align*}
 vec(\boldV_k)=\bigg[\bigg(-\widetilde{\Lambda}\otimes \boldI_n-\boldI_r\otimes \boldA\bigg)^{-1} \sum_{j=1}^m \widehat{\boldN}_j\otimes \boldN_j\bigg]^k\bigg(-\widetilde{\Lambda}\otimes \boldI_n-\boldI_r\otimes \boldA\bigg)^{-1}\Big(\widehat{\boldB}\otimes\boldB\Big)vec(\boldI_m) \nonumber\\
\end{align*}
 and similar result holds for the solution \boldW\ of the bilinear Sylvester equation
\begin{equation*}
\boldW(-\widetilde{\mathbf{\Lambda}})-\boldA^T\boldW-\sum \limits_{j=1}^m\boldN_j^T\boldW\widehat{\boldN}_j=\boldC^T\widehat{\boldC}.
\end{equation*}
\end{rem}
\subsection{Truncated Bilinear Iterative Rational Krylov Algorithm}
As in the case of Theorem \ref{thm:bilinear_H2_opt_breiten} and the resulting method B-IRKA,
the necessary conditions of Theorem \ref{h2conditions_truncated} lend itself perfectly to an iterative algorithm.
A reduced dimension bilinear system that generates a polynomial system which nearly satisfies these necessary conditions, i.e., (\ref{Sylvester_necessary_C})-(\ref{truncated_n_Sylvester_necessary}), can be constructed using 
Algorithm \ref{alg:tbrika} given below, which  we call \emph{truncated B-IRKA}, or TB-IRKA.
\begin{figure}[hh]
\begin{center}
    \framebox[5.125in][t]{
    \begin{minipage}[c]{5.0in}
    {
\begin{algorithm}[TB-IRKA]  \label{alg:tbrika}\\ 
\textbf{Input}: \boldA, $\boldN_1,\dots,\boldN_m$, \boldB, \boldC, \boldAr, $\boldNr_1,\dots,\boldNr_m$, \boldBr, \boldCr, N (the truncation index)\\
\textbf{Output}: $\boldAr$, $\boldNr_1,\dots,\boldNr_m$, $\boldBr$, $\boldCr$
\begin{enumerate}
\item \textbf{While}: Change in $\widetilde{\mathbf{\Lambda}} >tol$ do:
\item $\mathbf{R}\widetilde{\mathbf{\Lambda}} \mathbf{R}^{-1}=\boldAr$, $\widehat{\boldB}=\mathbf{R}^{-1}\boldBr$, $\widehat{\boldC}=\boldC \mathbf{R}$, for $k=1,\dots m$, $\widehat{\boldN}_k=\mathbf{R}^{-1}\boldNr_k\mathbf{R}$
\item Solve 
\[\boldV_1(-\mathbf{\Lambda})-\boldA\boldV_1=\boldB\widehat{\boldB}^T\]
\[\boldW_1(-\mathbf{\Lambda})-\boldA^T\boldW_1=\boldC^T\widehat{\boldC}^T\]
\item For $j=2,\ldots, N$, solve
\[\boldV_j(-\Lambda)-\boldA\boldV_j=\sum \limits_{k=1}^m\boldN_k\boldV_{j-1}\widehat{\boldN}_k^T\] 
and
\[\boldW_j(-\Lambda)-\boldA^T\boldW_j=\sum \limits_{k=1}^m\boldN_k^T\boldW_{j-1}\widehat{\boldN}_k\]
\item $\mathbf{S}_N=\Big(\sum \limits_{j=1}^N\boldV_j\Big)$, $\mathbf{U}_N=\Big(\sum \limits_{j=1}^{N}\boldW_j\Big)$.
\item $\boldAr=(\mathbf{U}_N^T\mathbf{S}_N)^{-1}\mathbf{U}_N^T\boldA\mathbf{S}_N$, $\boldNr_k=(\mathbf{U}_N^T\mathbf{S}_N)^{-1}\mathbf{U}_N^T\boldN_k\mathbf{S}_N$ for $k=1,\dots,m$,\\ $\boldBr=(\mathbf{U}_N^T\mathbf{S}_N)^{-1}\mathbf{U}_N^T\boldB$, $\boldCr=\boldC\mathbf{S}_N$.
\item \textbf{end while}
\end{enumerate}
\end{algorithm}
       }
    \end{minipage}
    }
    \end{center}
  \end{figure}
  
 The numerical advantage of TB-IRKA is apparent: Due to the special form of the ordinary Sylvester equations solved in steps 3.) and 4.) of Algorithm \ref{alg:tbrika}, TB-IRKA requires solving linear systems of dimension $n \times n$. This is in contrast to B-IRKA, which requires solving linear systems of dimension $(nr) \times (nr)$. The computational gains due to TB-IRKA will grow further as $r$ (and $n$) grows.

Upon convergence, the approximation constructed from  TB-IRKA yields a bilinear system that \emph{nearly} satisfies the $\mathcal{H}_2$ optimal necessary conditions for polynomial systems, and in the limit as $N$ approaches infinity, satisfies the necessary conditions exactly.  The following theorem makes explicit the sense in which the TB-IRKA approximations are asymptotically optimal.

\begin{thm}
Let $\zeta$ be a MIMO bilinear system with realization \boldA, $\boldN_1,\dots, \boldN_m$, $\boldB$, $\boldC$ of dimension $n$, and let $\widetilde{\zeta}^N$ be the polynomial system determined by the bilinear system with realization $\boldAr_N$, $\boldNr_{1,N},\dots,\boldNr_{m,N}$, $\boldBr_N$, $\boldCr_N$, where this realization is computed by TB-IRKA for the truncation index $N$.  Assume that the sequence $\{\widetilde{\zeta}^N\}_{N=1}^{\infty}$ converges strongly to a locally $\mathcal{H}_2$ optimal approximation $\widetilde{\zeta}$ with realization \boldAr, $\boldNr_1, \dots, \boldNr_m$, \boldBr, \boldCr.   Then the approximations $\widetilde{\zeta}^N$ satisfy
\begin{enumerate}[1.)]
\item$\text{trace}(\boldC\mathbf{S}_N\bolde_i\bolde_j^T)=\text{trace}(\boldCr_N\widetilde{\mathbf{S}}_N\bolde_i\bolde_j^T)+\epsilon_N $
\item$\text{trace}(\boldB^T\mathbf{U}_N\bolde_i\bolde_j^T)=\text{trace}(\boldBr_N^T\widetilde{\mathbf{U}}_N\bolde_i\bolde_j^T)+\epsilon_N$
\item$\mathbf{U}_N(:,i))^T\mathbf{S}_N(:,i)=\widetilde{\mathbf{U}}_N(:,i))^T\widetilde{\mathbf{S}}_N(:,i)^T+\epsilon_N$
\item$\mathbf{U}_N(:,i))^T\boldN_j\mathbf{S}_N(:,j)=\mathbf{U}_N(:,i))^T\boldNr_{j,N}\widetilde{\mathbf{S}}_N(:,j)+\epsilon_N$,
\end{enumerate}
where $ \lim \limits_{N \rightarrow \infty}\epsilon_N = 0$.  
\end{thm}

In this sense, the approximations are asymptotically optimal as $N$ approaches infinity.  From experience, we have found that $\epsilon_N$ decays very quickly for systems that are $\mathcal{H}_2$, so that even by the second or third term in the Volterra series, the remainder is negligible.

\begin{proof}
First observe that 
\begin{align}
\mathbf{S}_N(-\widetilde{\mathbf{\Lambda}}_N)-\boldA\mathbf{S}_N&=\sum \limits_{j=1}^m\boldN_j \big(\sum \limits_{k=1}^{N-1}\boldV_k \big) \widehat{\boldN}_{j,N}^T+\boldB\widehat{\boldB}_N^T\nonumber\\
\mathbf{S}_N(-\widetilde{\mathbf{\Lambda}}_N)-\boldA\mathbf{S}_N&=\sum \limits_{j=1}^m\boldN_j \mathbf{S}_N\widehat{\boldN}_{j,N}^T- \sum \limits_{j=1}^m \boldN_j \boldV_n \widehat{\boldN}_{j,N}^T+\boldB\widehat{\boldB}_N^T.\label{remainder_generalized_sylvester}
\end{align}
Consider the skew projection $\mathbf{\Pi}=\mathbf{S}_N(\mathbf{U}_N^T\mathbf{S}_N)^{-1}\mathbf{U}_N^T$.  Applying $\mathbf{\Pi}$ to both sides of the Sylvester equation (\ref{remainder_generalized_sylvester}) gives
\begin{align}
\mathbf{S}_N(\boldI_r(-\widetilde{\mathbf{\Lambda}}_N)-\boldAr_N\boldI_r-\sum \limits_{j=1}^m\boldNr_{j,N}\boldI_r\widehat{\boldN}_{j,N}-\sum \limits_{j=1}^m(\mathbf{U}_N^T\mathbf{S}_N)^{-1}\mathbf{U}_N^T\boldN_j\boldV_n\widehat{\boldN}_{j,N}^T-\boldBr_N\widehat{\boldB}_N)&=\mathbf{0} \nonumber\\
\Rightarrow
~\boldI_r(-\widetilde{\mathbf{\Lambda}}_N)-\boldAr_N\boldI_r-\sum \limits_{j=1}^m\boldNr_{j,N}\boldI_r\widehat{\boldN}_{j,N}-\sum \limits_{j=1}^m(\mathbf{U}_N^T\mathbf{S}_N)^{-1}\mathbf{U}_N^T\boldN_j\boldV_n\widehat{\boldN}_{j,N}^T-\boldBr_N\widehat{\boldB}_N&=\mathbf{0}.\label{reduced_remainder_sylvester}
\end{align}
Let $\widetilde{\mathbf{S}}_N$ be defined as the solution to the equation
\begin{equation}
\widetilde{\mathbf{S}}_N(-\widetilde{\mathbf{\Lambda}}_N)-\boldAr_N\widetilde{\mathbf{S}}_N=\sum \limits_{j=1}^m \boldNr_{j,N}\Big(\sum \limits_{k=1}^{N-1} \boldVr_k\Big)\widehat{\boldN}_{j_N}+\boldBr_N\widehat{\boldB}^T_N,
\end{equation}
where the matrices $\boldV_k \in \reals^{r \times r}$ are constructed iteratively in the same manner as the matrices $\boldV_k$ in steps 3.) and 4.) of TB-IRKA only in this case using the reduced-dimension parameters.
 Then
\begin{equation}\label{true_solution}
\widetilde{\mathbf{S}}_N(-\widetilde{\mathbf{\Lambda}}_N)-\boldAr_N\widetilde{\mathbf{S}}_N-\sum \limits_{j=1}^m\boldNr_{j,N}\widetilde{\mathbf{S}}_N\widehat{\boldN}_{j,N}^T+ \sum\limits_{j=1}^m\boldNr_{j,N}\boldVr_N\widehat{\boldN}_{j,N}^T-\boldBr_N\widehat{\boldB}_N^T=\mathbf{0}.
\end{equation}
Subtracting equation (\ref{true_solution}) from (\ref{reduced_remainder_sylvester}) gives
\begin{align*}
(\boldI_r-\widetilde{\mathbf{S}}_N)(-\widetilde{\mathbf{\Lambda}}_N)&-\boldAr_N(\boldI_r-\widetilde{\mathbf{S}}_N)-\sum \limits_{j=1}^m\boldNr_{j,N}(\boldI_r-\widetilde{\mathbf{S}}_N)\widehat{\boldN}_{j,N}^T\\
&+\sum \limits_{j=1}^m\Big((\mathbf{U}_N^T\mathbf{S}_N)^{-1}\mathbf{U}_N^T\boldN_{j,N}\boldV_n\widehat{\boldN}_{j,N}^T-\boldNr_{j,N}\boldVr_n\widehat{\boldN}_{j,N}^T\Big)=\mathbf{0}.
\end{align*}
From the assumption that $\widetilde{\zeta}_N \rightarrow \widetilde{\zeta}^*$ strongly it follows in a straightforward manner that $\lim \limits_{N \rightarrow \infty} \|\boldV_N\|=\lim \limits_{N \rightarrow \infty} \|\boldVr_N\|=0$.  Hence we have that $\lim \limits_{N \rightarrow \infty} \widetilde{\mathbf{S}}_N=\boldIr$.  An argument similar to the one given above yields $\lim \limits_{N \rightarrow \infty} \|\widetilde{\mathbf{U}}_N-\mathbf{S}_N^T\mathbf{U}_N\|=0$. Now let $\mathbf{\Gamma}_N=\widetilde{\mathbf{S}}_N-\boldI_r$.  To prove 1.) observe that 
\begin{align}
\text{trace}(\boldCr_N\widetilde{\mathbf{S}}_N\bolde_i\bolde_j^T) 
=\text{trace}(\boldC\mathbf{S}_N(\boldI_r+\mathbf{\Gamma}_N) \bolde_i\bolde_j^T).
\end{align}
Thus
\begin{align}
\text{trace}(\boldC\mathbf{S}_N\bolde_i\bolde_j^T)+\text{trace}(\boldC\mathbf{\Gamma}_N\bolde_i\bolde_j^T)
=\text{trace}(\boldC\mathbf{S}_N\bolde_i\bolde_j^T)+\epsilon_N.
\end{align}
Next we prove 3.). Let $\mathbf{\Theta}_N=\mathbf{S}_N^T\mathbf{U}_N$ and let $\mathbf{R}_N=\mathbf{\Theta}_N-\widetilde{\mathbf{U}}_N$.  Then
\begin{align}
\widetilde{\mathbf{U}}_N(:,i)^T&\boldNr_{u,N}\widetilde{\mathbf{S}}_N(:,j)\nonumber\\
=&\Big(\mathbf{\Theta}_N(:,i)^T+\mathbf{R}(:,i)^T\Big)\mathbf{\Theta}^{-1}\mathbf{U}_N^{T}\boldN_u\mathbf{S}_N(\boldI_r(:,j)+\mathbf{\Gamma}(:,j)_N)\nonumber \\
=&\mathbf{U}_N(:,i)^T\boldN_u\mathbf{S}_N(:,j)+\epsilon_N.
\end{align}
Equations 2.) and 4.) are proved in a similar manner.
\end{proof}

\subsection*{Reducing bilinear systems without  a convergent $\mathcal{H}_2$ norm}
It is not uncommon to encounter a bilinear system which does not have a convergent $\mathcal{H}_2$ norm.  For example, the bilinear system approximation to the nonlinear RC circuit model first introduced by Skoogh and Bai \cite{bai2006projection} is a standard benchmark model for testing methods of bilinear model reduction, but this model does not have a convergent $\mathcal{H}_2$ norm.  Other benchmark models, such as bilinear approximations to Burgers' equation are also not $\mathcal{H}_2$ for modest Reynolds numbers; see e.g. \cite{Breiten_H2,breiten2009ms}.   In these situations, there are a few options available.  One technique, as suggested by
\cite{Breiten_H2}, is to scale $\zeta$ by the mapping \[\gamma \mapsto \zeta_{\gamma}:= (\boldA,\text{ }\gamma \boldN_1,\ldots,\text{ }\gamma \boldN_m,\text{ }\gamma \boldb,\text{ } \boldc),\] where $\gamma<1$ is chosen sufficiently small so that $\|\zeta_{\gamma}\|_{ _{\mathcal{H}_2}}< \infty$.   $\mathcal{H}_2$ optimal model reduction is carried out on $\zeta_{\gamma}$, and the original input-output map can be recovered by scaling the inputs $u(t)$ for the original system by $u(t)/\gamma$.  
Frequently the scaling approach yields very accurate approximations for inputs of interest, but there are challenges.  If the system is large enough, it is costly to determine a good scaling parameter.  If $\gamma$ is chosen too small, then the scaled system $\zeta_\gamma$ may function as an essentially linear system and may destroy the advantages of doing model reduction in the original bilinear setting as opposed to the linearized version. 
Another approach is to match some combination of subsystem moments in the hopes of capturing the dominant portion of the Volterra series for the inputs of interest.  

In this paper we will propose another alternative. We note that 
any truncation $\zeta^N$ of the original system has a finite $\mathcal{H}_2^N$ norm. Computing an $\mathcal{H}_2^N$ optimal approximation to $\zeta^N$ using TB-IRKA is therefore another alternative when $\zeta$ is not $\mathcal{H}_2$.  Frequently an $\mathcal{H}_2$ optimal approximation of the first few terms in the Volterra series is sufficiently accurate to match the output of the $\zeta$. In Section \ref{sec:examples}, we will present an example to demonstrate this approach.

\section{Numerical Examples} \label{sec:examples}
In this section, we illustrate  the performance of the proposed method TB-IRKA using four numerical examples.
\subsection{Heat transfer model}
We consider a boundary controlled heat transfer system.  This model has  become a benchmark for testing model reduction methods, and it was first introduced in \cite{benner2011lyapunov}.  The system dynamics are governed by the heat equation subject to Dirichlet and Robin boundary conditions.
\begin{align*}
x_t&=\Delta x & \text{ in } (0,1) \times (0,1),\\
n \cdot \nabla x&=0.8 \cdot u_{1,2,3} (x-1) & \text{ on } \Gamma_1, \Gamma_2, \Gamma_3\\
x&=0.8 \cdot u_4& \text{ on } \Gamma_4,
\end{align*}
where $\Gamma_1, \Gamma_2, \Gamma_3$ and $\Gamma_4$ denote the boundaries of the unit square.  A carefully constructed spatial discretization using $k^2$ grid points yields a bilinear system of order $n=k^2$, with two inputs and one output, chosen to be the average temperature on the grid.  Taking $k=100$, we demonstrate TB-IRKA on a bilinear system of order $n=10,000$, and compare it with B-IRKA for the same system.  Figure \ref{figure:relative_errors_2_4_terms_TB_IRKA_and_BIRKA_heat_transfer} compares the relative $\mathcal{H}_2$ error in TB-IRKA approximations truncated at $N=2$ and $N=4$ terms with the relative error in the  B-IRKA approximation for the same orders.  The figure illustrates that  using $N=4$ terms in the Volterra series yields TB-IRKA approximations that are essentially equivalent to the B-IRKA approximations for all orders.  For $N=2$, their is still relatively little difference between the two approaches for the orders upto $r=16$.  Both B-IRKA and TB-IRKA started from the same initial guess. Next we compare the average time per iteration for all orders of approximation in Figure \ref{figure:average_time_per_iteration_10000_heat_transfer}.   For small reduced orders such as $r=2,4$, B-IRKA is marginally faster, however on average when $N=2$ there was a 62\% decrease in the time per iteration in TB-IRKA compared to B-IRKA and when $N=4$, there was a 30\% decrease in the time per iteration in TB-IRKA compared to B-IRKA.  
 \begin{figure}[hhhh]
\center
\includegraphics[scale=.5]{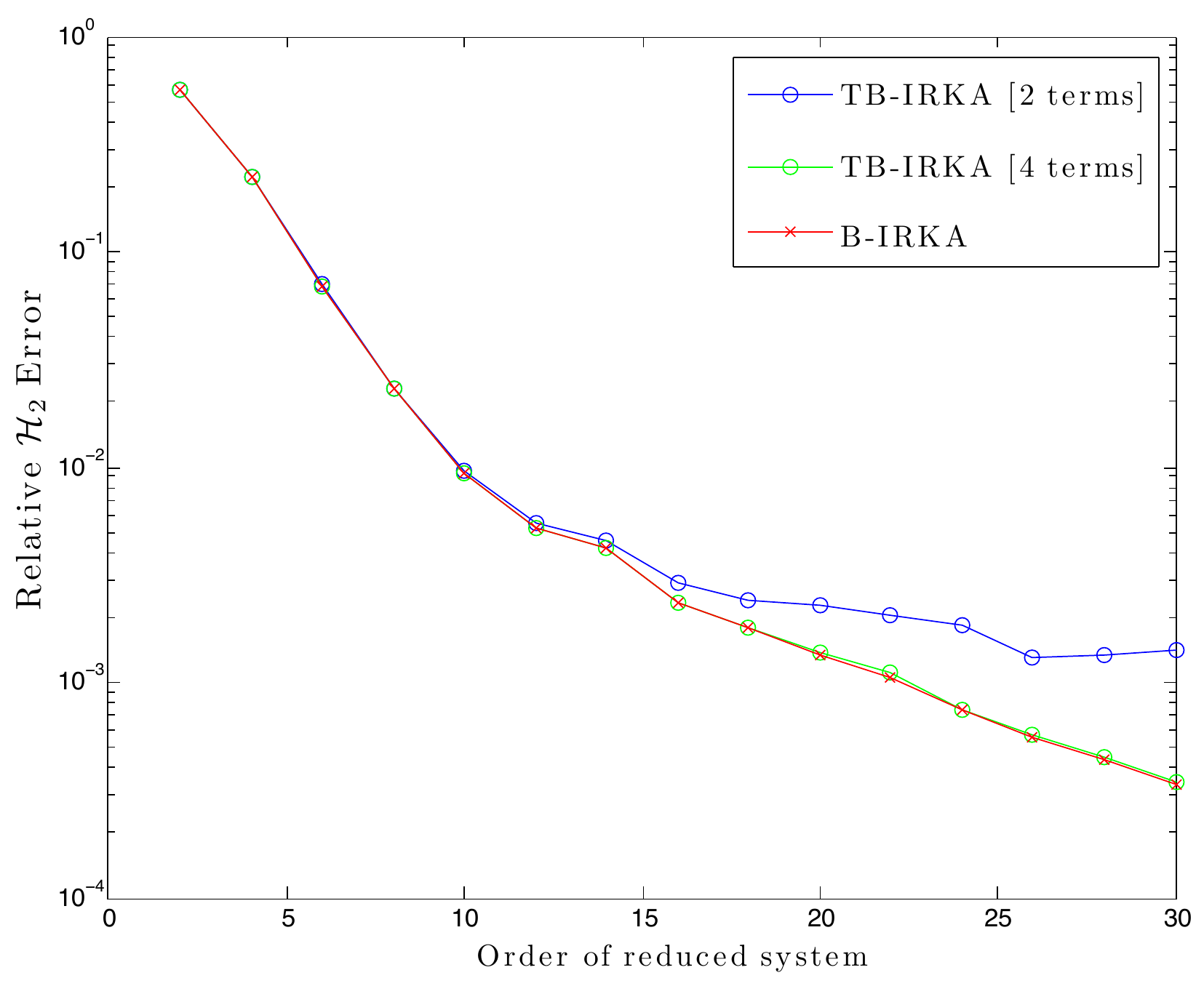}
\caption{Comparison of TB-IRKA and B-IRKA approximations of heat transfer control system}
\label{figure:relative_errors_2_4_terms_TB_IRKA_and_BIRKA_heat_transfer}
\end{figure}
 \begin{figure}[hhhh]
\center
\includegraphics[scale=.5]{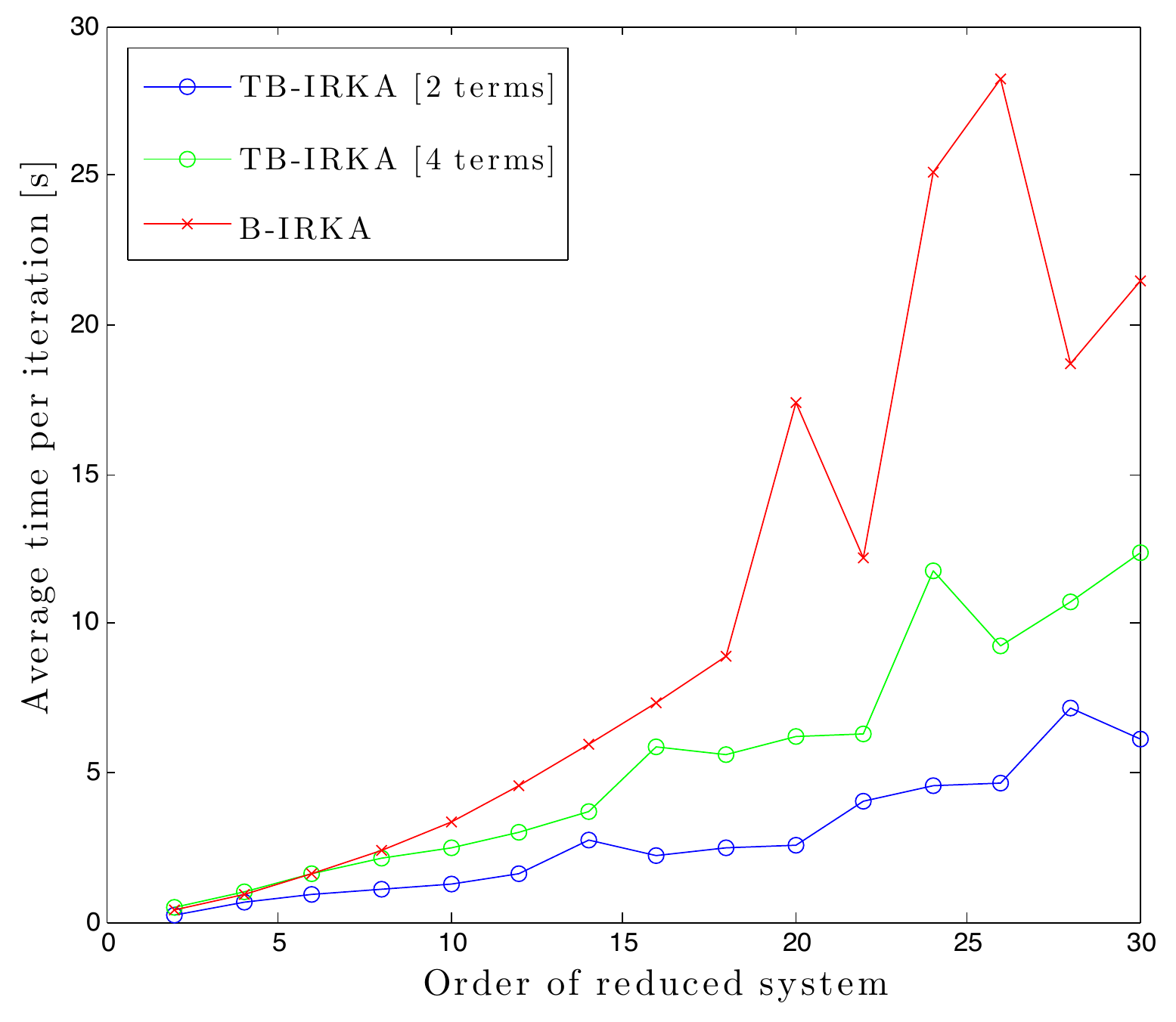}
\caption{Comparison of average time per iteration in TB-IRKA and B-IRKA for several orders}
\label{figure:average_time_per_iteration_10000_heat_transfer}
\end{figure}

\subsection{A bilinear model of the Fokker-Planck equations}
The following example is an application from stochastic control that was first introduced by Hartmann \emph{et. al} in \cite{hartmann_balanced} and later used as a test case for B-IRKA in \cite{Breiten_H2}. Consider a Brownian particle confined by a double-well potential $W(x)=(x^2-1)^2$.  Assume the particle is initially in the left well, and is dragged to the right well.  The particle's motion can be described by the stochastic differential equation
\[dX_t=-\nabla V(X_t, t)dt+ \sqrt{2\sigma}dW_t,\]
with $\sigma=2/3$ and $V(x, u)= W(x, t)+\Phi(x, u_t)=W(x)-xu-x$.  As an alternative to these equations it is noted in \cite{hartmann_balanced} that one can instead determine the underlying probability distribution function
\[\rho(x,t)dx=P[X_t \in [x, x+dx)]\]
which is described by the Fokker-Planck equation
\begin{align*}
\frac{\partial \rho}{\partial t}&=\sigma \Delta \rho+ \nabla \cdot ( \rho \nabla V), &(x, t) \in (a, b) \times (0, T],\\
0&=\sigma \nabla \rho+ \rho \nabla B, &( x, t) \in \{a, b\}\times [0, T],\\
\rho_0&=\rho,& (x,t) \in (a,b) \times 0
\end{align*}
A finite-difference discretization of the Fokker-Planck equations consisting of 500 nodes in the interval $[-2, 2]$ leads to a SISO bilinear system, where the output matrix $\boldc$ is a discretization of the (set-theoretic) characteristic function of the interval $[0.95, 1.05]$.  Figure \ref{figure:Fokker_Planck_13sub} compares the relative $\mathcal{H}_2$ error in the reduced order models computed from B-IRKA and TB-IRKA after truncating at the $13th$ term in the Volterra series. It was necessary to keep this many subsystems because the Volterra series for this model converged somewhat slowly, and so $\mathcal{H}_2$ error in the approximation decayed slowly as well.  As Figure \ref{figure:Fokker_Planck_13sub} demonstrates, TB-IRKA 
replicates the accuracy of B-IRKA very well  for most orders of approximation.  For the orders of approximation $r=2,4$, the average time per iteration of TB-IRKA and B-IRKA was the same, but as the reduced order system grew to between $r=6$ and $r=24$, the average time per iteration for TB-IRKA was 51\% less than for B-IRKA on the average.  Figure \ref{figure:Fokker_Planck_13sub_times} compares the average time per iteration for several reduced orders, illustrating that as $r$ increases,  so do the numerical gains in TB-IRKA.

\begin{figure}[hhhh]
\center
\includegraphics[scale=.6]{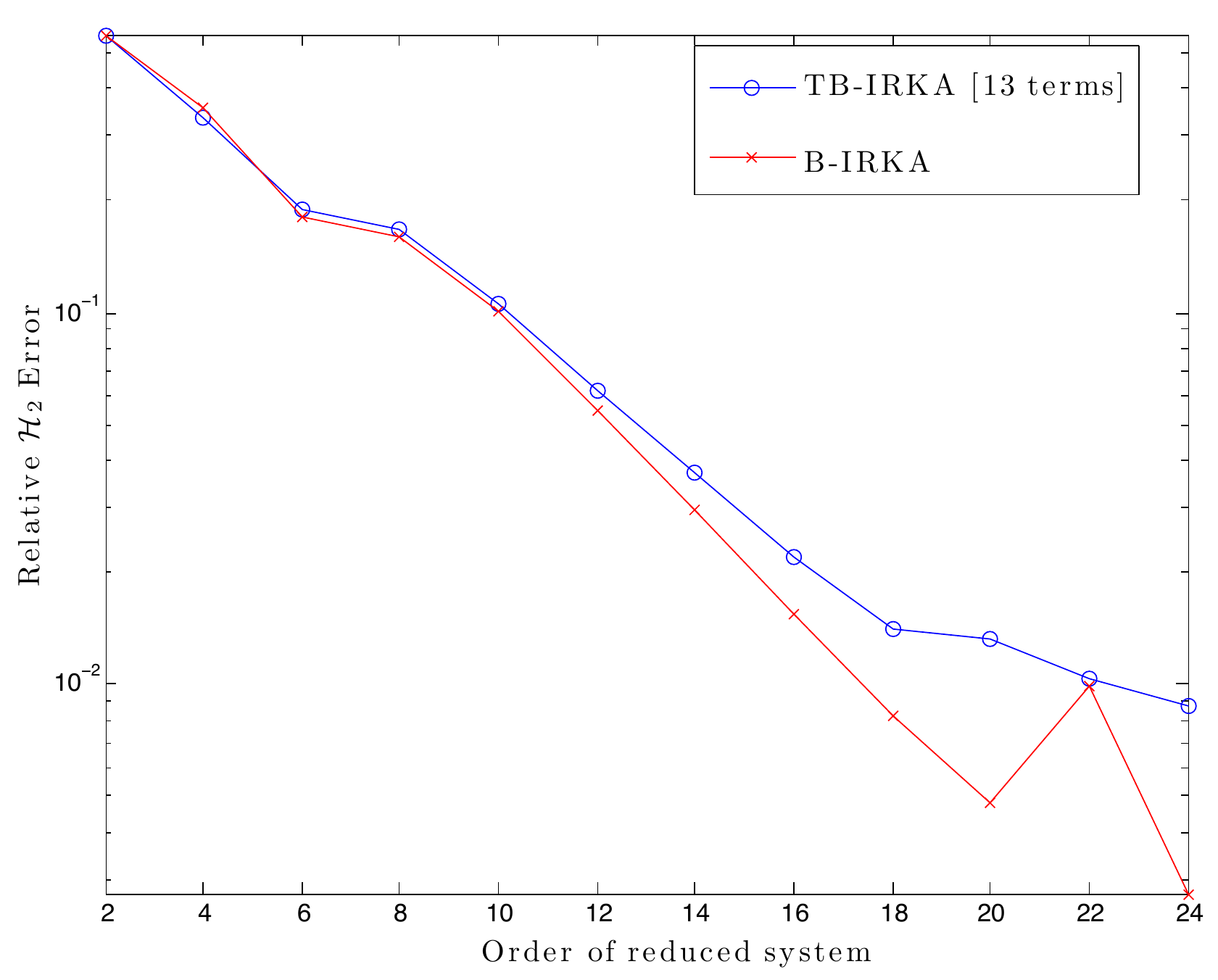}
\caption{Comparison of the relative $\mathcal{H}_2$ error for B-IRKA and TB-IRKA approximations to the Fokker-Planck system}
\label{figure:Fokker_Planck_13sub}
\end{figure}

\begin{figure}[hhhh]
\center
\includegraphics[scale=.6]{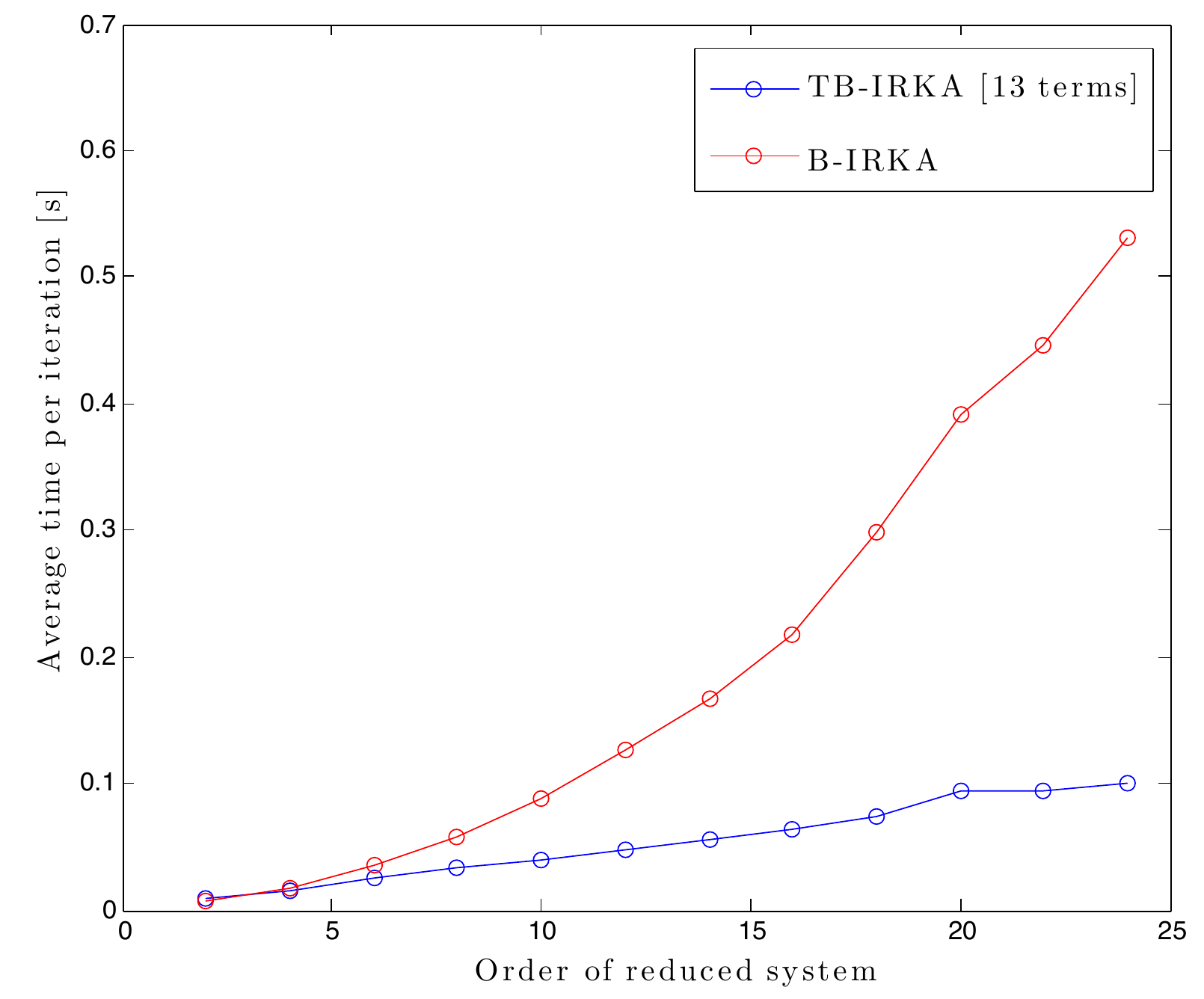}
\caption{Comparison of average time per iteration using B-IRKA and TB-IRKA[13 terms] for the Fokker-Planck system}
\label{figure:Fokker_Planck_13sub_times}
\end{figure}

\subsection{Viscous Burgers' Equation Control System}\label{section:VB}
Another model reduction benchmark originally introduced in \cite{breiten2009ms} is a bilinear control system derived from Burgers' equation.  Consider the viscous Burgers' equation
\begin{align*}
\frac{\partial v}{\partial t}+v\frac{\partial v}{\partial x}= \nu \frac{ \partial^{2} v}{\partial x^2}, & (x, t) \in (0, 1) \times (0, T)
\end{align*}
subject to initial and boundary conditions
\begin{align*}
v(x, 0)= 0, \hspace{20pt}x \in [0, 1],& \hspace{20pt}v(0,t)=u(t), \hspace{20pt}  v(1,t)=0 \hspace{10pt} t \ge 0
\end{align*}

Discretizing Burgers' equation in the spatial variable using $n_0$ nodes in a standard central difference finite difference scheme leads to a system of nonlinear ordinary differential equation where the nonlinearity is quadratic in the state. Measurements of the system are given as the spatial average of $v$. The Carleman linearization  technique applied to this system yields a bilinearized system of dimension $n=n_0+n_0^2$ that exactly matches the input-output behavior of the original nonlinear system, since the nonlinearity is only quadratic. Here we take $\nu=0.01$, corresponding to a Reynolds number of 100 and construct a bilinear system $\zeta$ of order $n=930$.  $\zeta$ is not an $\mathcal{H}_2$ system, which can be checked by observing that the series used to compute its control grammian diverges.  For this example we compare TB-IRKA using the truncation index $N=2$ with the scaled version of B-IRKA.  An order $r=9$ approximation is used to compute the response of both methods to the inputs $u(t)=e^{-t}$ and $u(t)=\sin(20t)$.  The relative error in the output of using the scaling values $\gamma=0.4$ and $\gamma=0.5$  for B-IRKA are compared with the TB-IRKA approximation in Figures \ref{figure:relative_response_error_non_h2_burger_decaying_exponential}, \ref{figure:relative_response_errors_high_oscillating_sine_wave}.  As the figures show, very good approximation results using B-IRKA can be obtained for the right value of $\gamma$; in this case $\gamma=0.4$ yielded good approximations, but the quality of the approximations is fairly sensitive to the choice of $\gamma$ as $\gamma=0.5$ resulted in a poor approximation. On the other hand,  for both inputs the TB-IRKA approximation yields a highly accurate approximation, and indeed, for the input $u(t)=e^{-t}$, yields a smaller $L_\infty$ output error $\big\| \mathbf{y} - \widetilde{\mathbf{y}}\big\|_{L_\infty}$ than B-IRKA.

\begin{figure}[hhhh]
\center
\includegraphics[scale=.5]{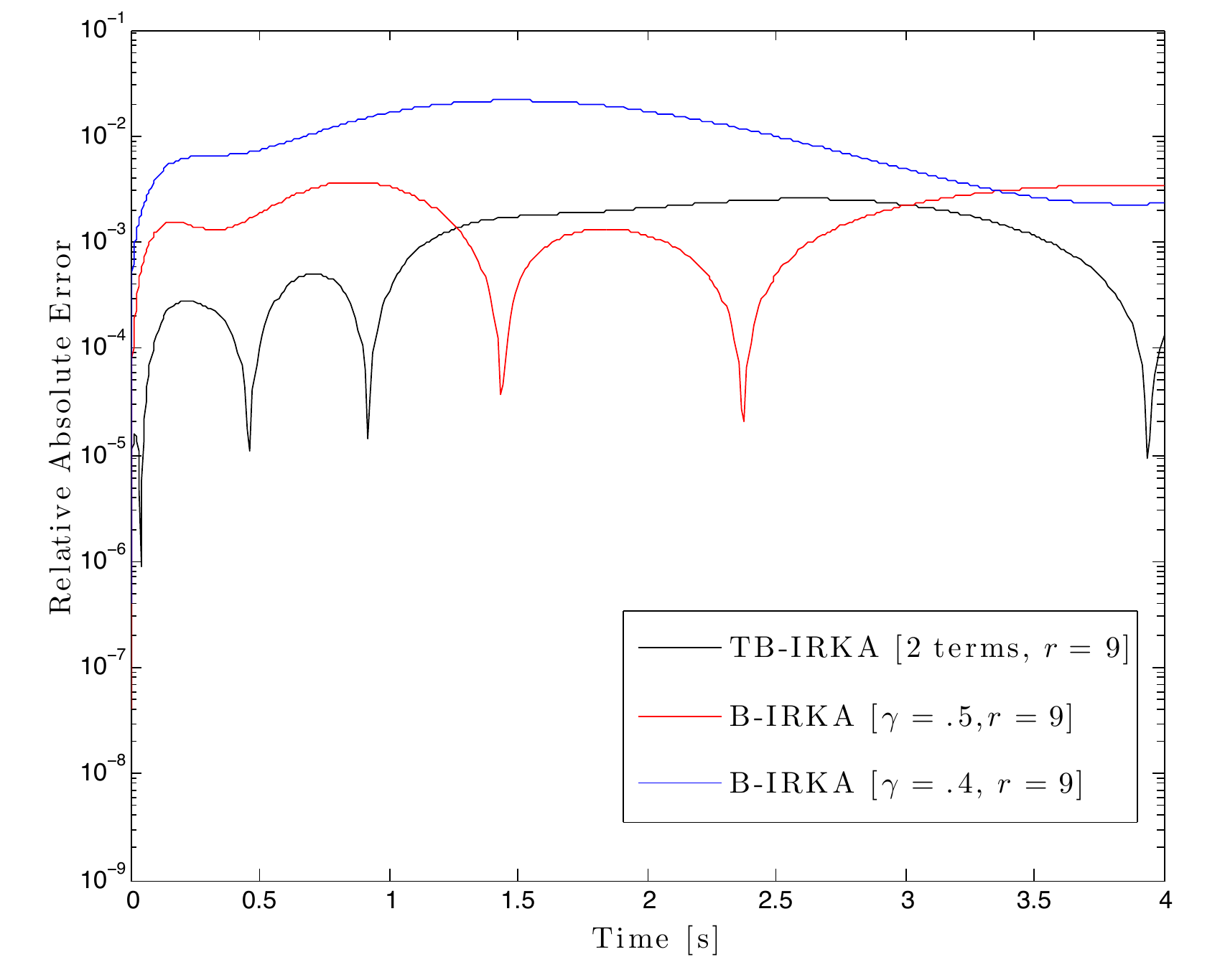}
\caption{\textbf{Burgers' Equation}: A comparison of the  TB-IRKA and scaled B-IRKA output error 
$\big\| \mathbf{y} - \widetilde{\mathbf{y}}\big\|_{L_\infty}$
for the input $u(t)=e^{-t}$. }
\label{figure:relative_response_error_non_h2_burger_decaying_exponential}
\end{figure}

\begin{figure}[hhhh]
\center
\includegraphics[scale=.5]{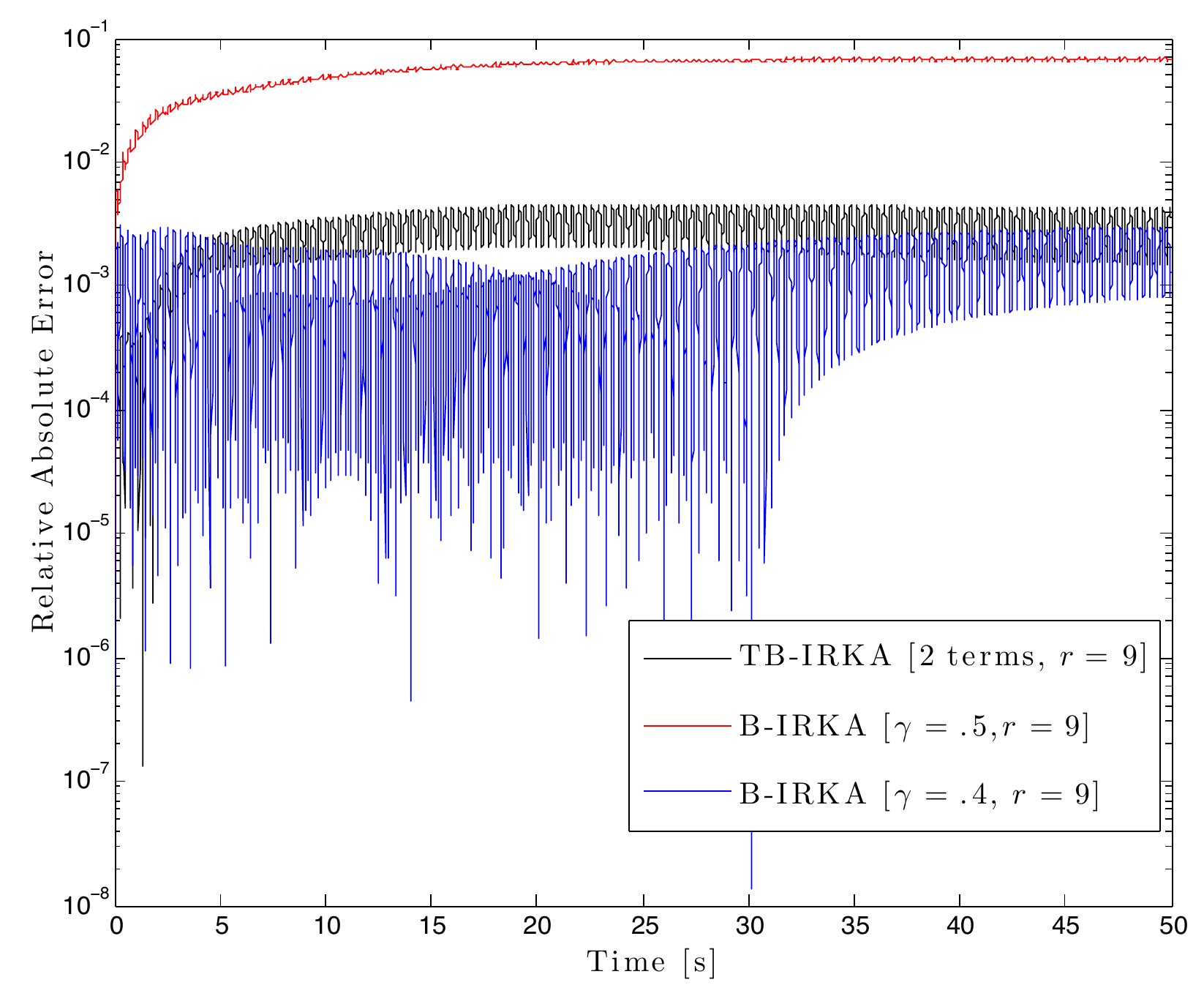}
\caption{\textbf{Burgers' Equation}: A comparison of the TB-IRKA and scaled B-IRKA output error $\big\| \mathbf{y} - \widetilde{\mathbf{y}}\big\|_{L_\infty}$ for the input $u(t)=\sin(20t).$ }
\label{figure:relative_response_errors_high_oscillating_sine_wave}
\end{figure}

\subsection{A parameter-varying convection-diffusion problem}
Benner and Breiten showed  \cite{benner2011h2} that certain classes of parameter-varying linear systems can be effectively approximated over the desired range of parameters by appropriately reformulating the linear system  as a bilinear system.  Here we carry out this approach for a parameter-varying convection-diffusion problem from \cite{baur2011interpolatory}.  The model is governed by the standard convection-diffusion equations
\begin{align*}
&\frac{\partial \boldsymbol{x} }{\partial t}(t, \boldsymbol{\xi})=p_0\Delta\boldsymbol{x}(t, \boldsymbol{\xi})+\sum \limits_{i=1}^2 p_i\nabla \boldsymbol{x}(t,\boldsymbol{\xi})+\boldsymbol{b}(\boldsymbol{\xi})u(t),\\ &\boldsymbol{\xi} \in [0,1]\times[0,1], t\in (0, \infty),&\hspace{-70pt}\boldsymbol{x}(t,\boldsymbol{\xi})=0, \xi \in \partial ([0,1]\times[0,1])
\end{align*}
and the parameters $p_0, p_1, p_2$ need to be adjusted to capture the particular physics that is being modeled.  
After a finite-difference discretization in the spatial variable $\xi$, we obtain the linear parameter-varying dynamical system
\begin{align}
\dot{\boldsymbol{x}}(t)&=p_0\boldA_0\boldsymbol{x}(t)+\sum \limits_{i=1}^2\boldA_i\boldsymbol{x}(t) p_i+\boldb u(t) \label{LPV_sys}\\
y&=\boldc\boldsymbol{x}(t) \nonumber,
\end{align}
where $\boldc=\bolde_n$ is chosen as the observation matrix. This system can be viewed as a bilinear system where the parameters $p_1$ and $p_2$ are particular system inputs.  We can rewrite system (\ref{LPV_sys}) as a bilinear system with three inputs and one output:

\begin{align*}
\dot{\boldsymbol{x}}&=\boldsymbol{A}\boldsymbol{x}+\sum \limits_{k=1}^3\boldN_k\boldsymbol{x}u_k(t)+\boldB\boldsymbol{u}(t)\\
y(t)&=\boldc\boldsymbol{x}(t)
\end{align*}
 with \boldA=$p_0 \boldA_0$, $\boldN_1=\boldA_1$, $\boldN_2=\boldA_2$, $\boldN_3=\boldsymbol{0}$, $\boldB=[\boldsymbol{0}, \boldb]\in \reals^{n\times 3}$, for inputs  of interest $\boldsymbol{u}(t)=[ p_1,p_2, u(t)]^T$.  

The parameter range of interest is $p_0\in[0.1, 1]$, $p_1,p_2 \in [0, 1]$.  Taking $p_0=1$, we compute TB-IRKA approximations keeping $2, 3,$ and $6$ terms in the Volterra series, and compare them with the B-IRKA approximation to the full bilinear system.  Each of these approximations are of dimension  $r=12$.  To place the reduced bilinear system matrices back into the linear parameter-varying formulation we use $p_0 \widetilde{\boldA}_0 = \boldAr \in \reals^{r \times r}$, $\widetilde{\boldA}_1 = \boldNr_1 \in \reals^{r\times r}$ and $\widetilde{\boldA}_2 = \boldNr_2 \in \reals^{r\times r}$ as the reduced-dimension matrices that approximate the linear parameter-varying system (\ref{LPV_sys}).    In order to evaluate the accuracy of the approximations,  we vary the parameters $p_1$ and $p_2$ over the whole parameter range of interest, and for each selection of parameters we compute the relative $\mathcal{H}_2$ norm of the error between the full and reduced dimension systems for that choice of parameters.  The surfaces plotted in Figure \ref{figure:p_0_is_one_comparison} show how the relative $\mathcal{H}_2$ error of the linear systems varies over the parameter values.  As Figure  \ref{figure:p_0_is_one_comparison} shows, TB-IRKA with $N=2$ actually gives the best approximation error over the parameter space, and the approximation error increases as the number of terms kept in the Volterra series increases; with B-IRKA giving, in this case, the largest errors over the parameter space.  We believe this is due to the fact B-IRKA is actually a better approximation over the whole $\mathcal{L}_2$ unit ball of inputs for the bilinear reduction, and thus it gives up  accuracy for these particular inputs once converted back to the parametric linear system. Moreover,
we certainly do {\it not} claim that TB-IRKA will always yield smaller error for reducing parametric linear models. We note that regardless, all four reduced models give very accurate approximations with relative errors in the order of $10^{-4}$.   Next 
we take $p_0=0.5$ and compute two reduced models of order $r=12$ using TB-IRKA with $N=2$ and B-IRKA approximation, both of dimension $r=12$.  Figure \ref{figure:p_0_is_point_five_comparison} shows the relative $\mathcal{H}_2$ error in the linear systems over the parameter range for $p_0=0.5$.   Again for this case, TB-IRKA yields a smaller approximation error than B-IRKA, and both yield nearly uniform error over the range of parameters.  
\begin{figure}[hhhh]
\center
\includegraphics[scale=.4]{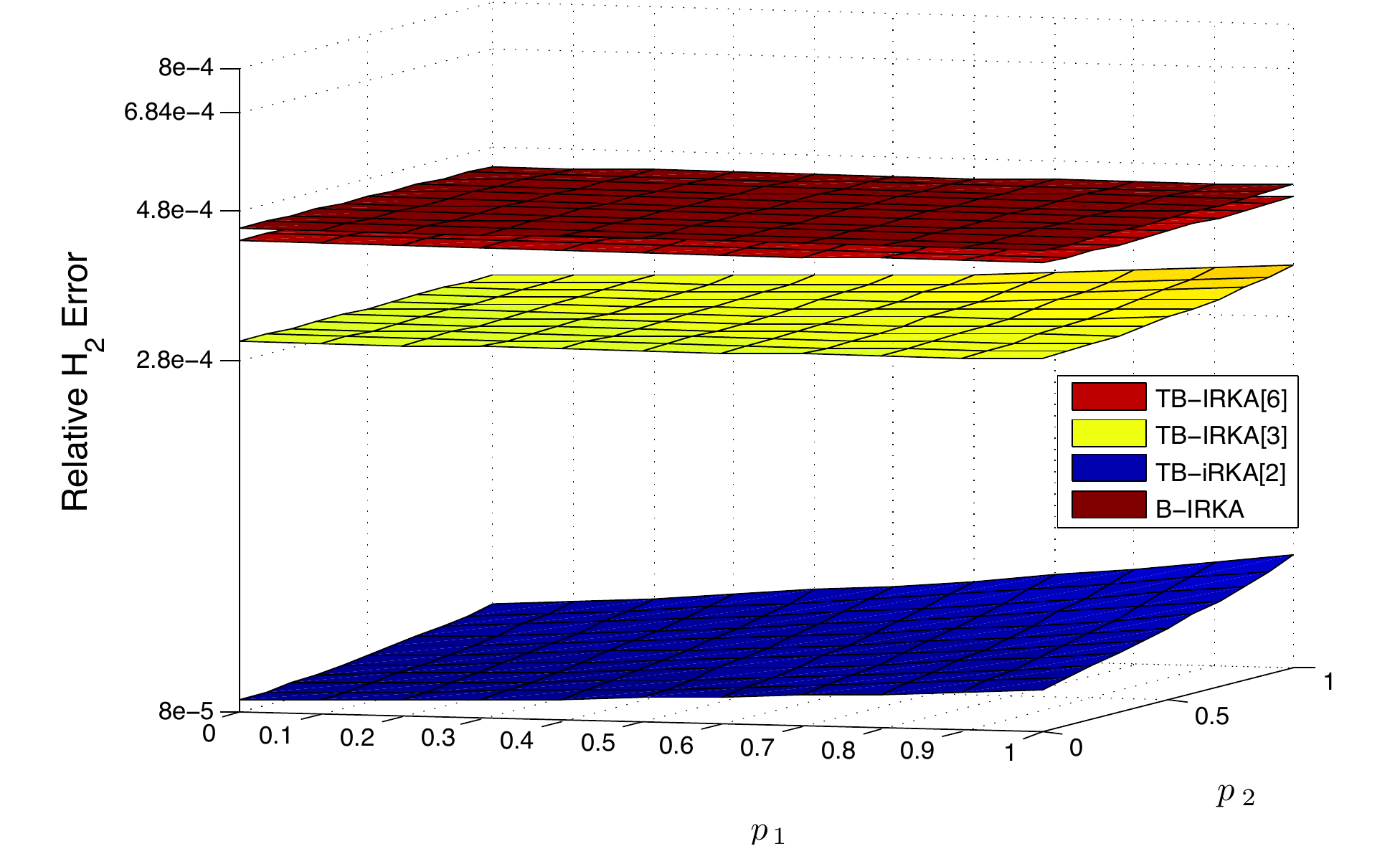}
\caption{\textbf{Convection-diffusion problem}:  Comparison of the relative $\mathcal{H}_2$ error in the B-IRKA and TB-IRKA[2, 3 and 6 terms] approximations taking $p_0=1$ and varying over the parameter range for $p_1$ and $p_2$}
\label{figure:p_0_is_one_comparison}
\end{figure}

\begin{figure}[hhhh]
\center
\includegraphics[scale=.35]{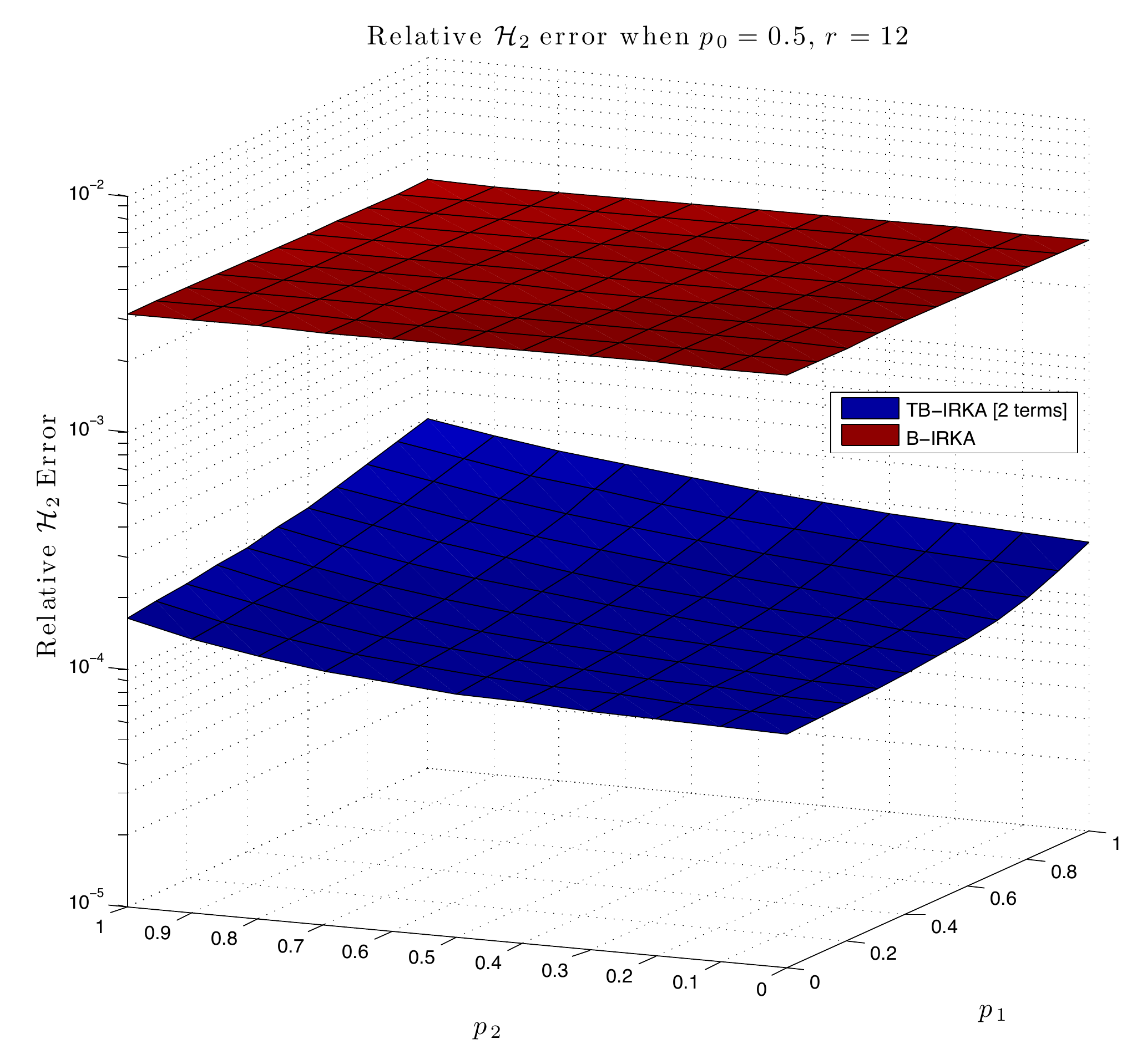}
\caption{\textbf{Convection-diffusion problem}:  Comparison of the relative $\mathcal{H}_2$ error in the B-IRKA and TB-IRKA[2 terms] approximations taking $p_0=0.5$ and varying over the parameter range for $p_1$ and $p_2$.}
\label{figure:p_0_is_point_five_comparison}
\end{figure}

\section{Conclusions} \label{sec:conc}
We have introduced an interpolation framework for model reduction of large-scale bilinear systems where reduced model enforces multipoint interpolation of the underlying Volterra series as opposed to interpolating some of the leading subsystem transfer functions as done in the existing approaches. We show that this new interpolation framework
is directly related to optimal $\mathcal{H}_2$ model reduction of bilinear systems as we proved that
the $\mathcal{H}_2$ requires multivariate Hermite interpolation in terms of the  Volterra series framework. 
Finally,  based on the multipoint interpolation on the truncated Volterra series representation, we  have introduced a model reduction algorithm leading to an asymptotically optimal approach to $\mathcal{H}_2$ optimal model reduction of bilinear systems. Several numerical examples  demonstrate the effectiveness of the proposed  approach.

\section{Acknowledgements}
The authors thank Prof. P. Benner and Dr. T. Breiten for providing the data for several numerical examples and also their Matlab implementation of B-IRKA. 

\bibliographystyle{siam}
\bibliography{references}

\setcounter{section}{0}
 \renewcommand{\thesection}{\Alph{section}}
 \section{Proof of Theorem \ref{h2conditions_truncated}} \label{appendix_proof}
 The proof follows by differentiating the error expression $E_N$ with respect to the parameters $\widetilde{\mathbf{\Lambda}}, \widehat{\boldN}, \widehat{\boldb}, \widehat{\boldc}$ and making use of Lemma \ref{lemma:product_rule} (taking $\mathbf{G}_k=\mathbf{0}$ for $k=1,\dots, m$ and $\mathbf{K}=\boldN$ in Lemma \ref{lemma:product_rule}) and the permutation matrix $\mathbf{M}$ in (\ref{permutation_matrix}). 
We start with differentiating $E_N$ with respect to the entries of $\widehat{\boldC}$ to obtain
 \begin{align}
 \frac{\partial E_N}{\partial \widehat{\boldC}_{i,j}}= 2vec(\boldI_p)^T \Big([\mathbf{0}, -\bolde_i\bolde_j^T]&\otimes[\boldC -\widehat{\boldC}]\Big)  \sum \limits_{k=0}^N\Big[\Big(-\cbfAhat\otimes\cbfI-\cbfI\otimes\cbfAhat\Big)^{-1} \Big(\sum \limits_{j=1}^m  \cbfNhat_j\otimes\cbfNhat_j\Big)\Big]^k\nonumber\\
&\times \Big(-\cbfAhat\otimes\cbfI-\cbfI\otimes\cbfAhat\Big)^{-1} \begin{bmatrix}\boldB\\ \widehat{\boldB}\end{bmatrix} \otimes \begin{bmatrix}\boldB\\ \widehat{\boldB}\end{bmatrix}vec(\boldI_m), \nonumber \\
= 2vec(\boldI_p)^T \Big([\mathbf{0}, -\bolde_i\bolde_j^T]&\otimes[\boldC -\widetilde{\boldC}]\Big)  \sum \limits_{k=0}^N\Big[\Big(-\cbfAhat\otimes\cbfI-\cbfI\otimes\cbfAtilde\Big)^{-1} \Big(\sum \limits_{j=1}^m  \cbfNhat_j\otimes\cbfNtilde_j\Big)\Big]^k\nonumber\\
&\times \Big(-\cbfAhat\otimes\cbfI-\cbfI\otimes\cbfAtilde\Big)^{-1} \begin{bmatrix}\boldB\\ \widehat{\boldB}\end{bmatrix} \otimes \begin{bmatrix}\boldB\\ \widetilde{\boldB}\end{bmatrix}vec(\boldI_m), \nonumber
\end{align}
where 
$$
\widehat{\cbfA\,} = \begin{bmatrix}\boldA& \mathbf{0}\\ \mathbf{0} & \widetilde{\mathbf{\Lambda}}\end{bmatrix},\,
\widetilde{\cbfA\,} = \begin{bmatrix}\boldA& \mathbf{0}\\ \mathbf{0} & \widetilde{\mathbf{\boldA}}\end{bmatrix},\,
\cbfI = \begin{bmatrix}\boldI_n& \mathbf{0}\\ \mathbf{0} & \boldI_r\end{bmatrix},\,
\widehat\cbfN_j = \begin{bmatrix}\boldN_j& \mathbf{0}\\ \mathbf{0}&\widehat{\boldN}_j\end{bmatrix},\,
\,{\rm and}~\widetilde\cbfN_j = \begin{bmatrix}\boldN_j& \mathbf{0}\\ \mathbf{0}&\widetilde{\boldN}_j\end{bmatrix}.
$$
Continuing these tedious manipulations leads to
\begin{align}
 \frac{\partial E_N}{\partial \widehat{\boldC}_{i,j}} =2vec(\boldI_p)^T \Big( -\bolde_i\bolde_j^T\otimes[\boldC -\boldCr]\Big)&  \sum \limits_{k=0}^N\Big[\Big(-\widetilde{\mathbf{\Lambda}}\otimes \cbfI-\boldI_r\otimes\cbfAtilde\Big)^{-1}\Big(\sum \limits_{j=1}^m\widehat{\boldN}_j\otimes\cbfNtilde_j\Big)\Big]^k \nonumber\\
&\times\Big(-\widetilde{\mathbf{\Lambda}}\otimes\cbfI-\boldI_r\otimes\cbfAtilde\Big)^{-1}\widehat{\boldB}\otimes \begin{bmatrix}\boldB\\ \boldBr\end{bmatrix}vec(\boldI_m)\nonumber\\
=2vec(\boldI_p)^T \Big( -\bolde_i\bolde_j^T\otimes[\boldC -\boldCr]&\Big) \mathbf{M}^T \sum \limits_{k=0}^N\Big[\big(\cbfI_{\boldsymbol{\Lambda}}-\cbfI_\boldA\big)^{-1}\nonumber
\Big( \sum\limits_{j=1}^m \cbfM_j\Big)\Big]^k\\
&
\times\big(\cbfI_{\boldsymbol{\Lambda}}-\cbfI_\boldA\big)^{-1}
\mathbf{M}\Bigg( \widehat{\boldB}\otimes \begin{bmatrix}\boldB\\ \boldBr\end{bmatrix}\Bigg)vec(\boldI_m),\nonumber
\end{align}
where 
$$
\cbfI_{\boldsymbol{\Lambda}} = 
\begin{bmatrix}-\widetilde{\mathbf{\Lambda}}\otimes\boldI_n& \mathbf{0}\\ \mathbf{0} &-\widetilde{\mathbf{\Lambda}}\otimes \boldI_r\end{bmatrix},\,
\cbfI_\boldA = 
\begin{bmatrix}\boldI_r\otimes \boldA& \mathbf{0}\\ \mathbf{0} & \mathbf{I}_r\otimes \boldAr\end{bmatrix},\,
{\rm and}~
\cbfM_j = 
\begin{bmatrix}\widehat{\boldN}_j\otimes \boldN_j& \mathbf{0}\\ \mathbf{0}&\widehat{\boldN}_j\otimes\boldNr_j\end{bmatrix}.
$$
Seperating the $\boldC$ and $\boldCr$ terms finally leads to
\begin{align}
 \frac{\partial E_N}{\partial \widehat{\boldC}_{i,j}}=-2vec(\boldI_p)^T&(\bolde_i\bolde_j^T \otimes \boldC) \sum \limits_{k=0}^N\bigg[\bigg(-\widetilde{\boldsymbol{\Lambda}}\otimes \boldI_n-\boldI_r\otimes \boldA\bigg)^{-1} \sum_{j=1}^m \widehat{\boldN}_j\otimes \boldN_j\bigg]^k\nonumber\\
&\times\bigg(-\widetilde{\boldsymbol{\Lambda}}\otimes \boldI_n-\boldI_r\otimes \boldA\bigg)^{-1}\Big(\widehat{\boldB}\otimes\boldB\Big)vec(\boldI_m) \nonumber\\
+2vec(\boldI_p)&(\bolde_i\bolde_j^T \otimes \boldCr) \sum \limits_{k=0}^N \bigg[\bigg(-\widetilde{\boldsymbol\Lambda}\otimes \boldI_r-\boldI_r\otimes \boldAr\bigg)^{-1} \sum \limits_{j=1}^m\widehat{\boldN}_j\otimes \boldNr_j\bigg]^k\nonumber\\
&\times\bigg(-\widetilde{\boldsymbol\Lambda}\otimes \boldI_n-\boldI_r\otimes \boldA\bigg)^{-1}\Big(\widehat{\boldB}\otimes\boldBr\Big)vec(\boldI_m). \label{truncated_differentiate_wrt_c}
\end{align}
Setting expression (\ref{truncated_differentiate_wrt_c}) equal to zero, we arrive at the necessary condition
\begin{align*}
vec(\boldI_p)^T(\bolde_i\bolde_j^T \otimes \boldC) \sum \limits_{k=0}^N&\bigg[\bigg(-\widetilde{\boldsymbol\Lambda}\otimes \boldI_n-\boldI_r\otimes \boldA\bigg)^{-1} \sum_{j=1}^m \widehat{\boldN}_j\otimes \boldN_j\bigg]^k
\\
&\times \bigg(-\widetilde{\boldsymbol\Lambda}\otimes \boldI_n-\boldI_r\otimes \boldA\bigg)^{-1}\Big(\widehat{\boldB}\otimes\boldB\Big)vec(\boldI_m) \nonumber\\
=vec(\boldI_p)^T(\bolde_i\bolde_j^T &\otimes \boldCr) \sum \limits_{k=0}^N \bigg[\bigg(-\widetilde{\boldsymbol\Lambda}\otimes \boldI_r-\boldI_r\otimes \boldAr\bigg)^{-1} \sum \limits_{j=1}^m\widehat{\boldN}_j\otimes \boldNr_j\bigg]^k \\
&\times\bigg(-\widetilde{\boldsymbol\Lambda}\otimes \boldI_n-\boldI_r\otimes \boldA\bigg)^{-1}\Big(\widehat{\boldB}\otimes\boldBr\Big)vec(\boldI_m).
\end{align*}
Unwinding these Kronecker product expressions, leads to (\ref{Sylvester_necessary_C}). 

Next, we first define 
$
\cbfE = \begin{bmatrix}\mathbf{0} &\mathbf{0}\\ \mathbf{0} & \bolde_i \bolde_i^T\end{bmatrix}
$
and then differentiate $E_N$ with respect to the entries of $\widetilde{\boldsymbol\Lambda}$ to obtain
\begin{align}
\frac{\partial E_N}{\partial \widetilde{\lambda}_i}=vec(\boldI_p&)^T\Big([\boldC -\widehat{\boldC}]\otimes[\boldC -\widehat{\boldC}]\Big)  \sum \limits_{k=0}^N \sum \limits_{l=0}^{k}\Big[\Big(-\cbfAhat\otimes\cbfI-\cbfI\otimes\cbfAhat\Big)^{-1}\sum \limits_{j=1}^m\cbfNhat_j\otimes\cbfNhat_j\Big]^l\nonumber\\
&\times\Big(-\cbfAhat\otimes\cbfI-\cbfI\otimes\cbfAhat\Big)^{-1}\bigg(\cbfE \otimes \cbfI\bigg) 
\Big[\Big(-\cbfAhat\otimes\cbfI-\cbfI\otimes\cbfAhat\Big)^{-1}\sum\limits_{j=1}^m\cbfNhat_j\otimes\cbfNhat_j\Big]^{k-l}\nonumber\\
&\times\Big(-\cbfAhat\otimes\cbfI-\cbfI\otimes\cbfAhat\Big)^{-1}\begin{bmatrix}\boldB\\ \widehat{\boldB}\end{bmatrix}\otimes\begin{bmatrix}\boldB\\ \widehat{\boldB}\end{bmatrix}vec(\boldI_m).\nonumber
\end{align}
Further simplifications and manipulations yield 
\begin{align}
\frac{\partial E_N}{\partial \widetilde{\lambda}_i}&=vec(\boldI_p)^T\Big([\boldC -\widehat{\boldC}]\otimes[\boldC -\boldCr]\Big)  \sum \limits_{k=0}^N \sum \limits_{l=0}^{k}\Big[\Big(-\cbfAhat\otimes\cbfI-\cbfI\otimes\cbfAtilde\Big)^{-1}\sum \limits_{j=1}^m\cbfNhat_j\otimes\cbfNtilde_j\Big]^l\nonumber\\
&\qquad\times\Big(-\cbfAhat\otimes\cbfI-\cbfI\otimes\cbfAtilde\Big)^{-1}\big(\cbfE \otimes \cbfI\big)\Big[\Big(-\cbfAhat\otimes\cbfI-\cbfI\otimes\cbfAhat\Big)^{-1}\sum \limits_{j=1}^m\cbfNhat_j\otimes\cbfNtilde_j\Big]^{k-l}\nonumber\\
&\qquad\times\Big(-\cbfAhat\otimes\cbfI-\cbfI\otimes\cbfAtilde\Big)^{-1}\begin{bmatrix}\boldB\\ \widehat{\boldB}\end{bmatrix}\otimes\begin{bmatrix}\boldB\\ \boldBr\end{bmatrix}vec(\boldI_m)\nonumber\\
=&vec(\boldI_p)^T\Big(-\widehat{\boldC}\otimes[\boldC -\boldCr]\Big)  \sum \limits_{k=0}^N \sum \limits_{l=0}^{k}\Bigg[\Bigg(-\widetilde{\mathbf{\boldsymbol\Lambda}}\otimes\cbfI-\boldI_r\otimes\cbfAtilde\Bigg)^{-1}\sum \limits_{j=1}^m\widehat{\boldN}_j\otimes\cbfNtilde_j\Bigg]^l\nonumber\\
\displaybreak[0]
&~~\times\Big(-\widetilde{\mathbf{\Lambda}}\otimes\cbfI- \boldI_r\otimes\cbfAtilde\Big)^{-1}\big(\bolde_i \bolde_i^T\otimes \cbfI\big)\Big[\Big(-\widetilde{\mathbf{\boldsymbol\Lambda}}\otimes\cbfI-\boldI_r\otimes\cbfAtilde\Big)^{-1}\sum_{j=1}^m\widehat{\boldN}_j\otimes\cbfNtilde_j\Big]^{k-l}\nonumber\\
&~~\times\Big(-\widetilde{\mathbf{\boldsymbol\Lambda}}\otimes\cbfI-\boldI_r\otimes\cbfAtilde\Big)^{-1}\big(\widehat{\boldB}\otimes\begin{bmatrix}\boldB\\ \boldBr\end{bmatrix}\big)vec(\boldI_m).\nonumber
\end{align}
Then, we employ the properties of the permutation matrix $\mathbf{M}$ to obtain  
\begin{align}
\frac{\partial E_N}{\partial \widetilde{\lambda}_i}=vec&(\boldI_p)^T\Big(-\widehat{\boldC}\otimes[\boldC -\boldCr]\Big) \mathbf{M}^T \nonumber \\ 
&\times \sum \limits_{k=0}^N \sum \limits_{l=0}^{k}\Big[\Big(\cbfI_{\boldsymbol{\Lambda}}-\cbfI_{\boldA}\Big)^{-1} \sum \limits_{j=1}^m\cbfM_j\Big]^l
\Big(\cbfI_{\boldsymbol{\Lambda}}-\cbfI_{\boldA}\Big)^{-1}\nonumber
\big(\cbfI_\bolde\big) \Big[\Big(\cbfI_{\boldsymbol{\Lambda}}-\cbfI_{\boldA}\Big)^{-1} \sum_{j=1}^m\cbfM_j\Big)\Big]^{k-l}
\nonumber\\
&\times \Big(\cbfI_{\boldsymbol{\Lambda}}-\cbfI_{\boldA}\Big)^{-1}\begin{bmatrix}\widehat{\boldB}\otimes\boldB\\ \widehat{\boldB}\otimes\boldBr\end{bmatrix}vec(\boldI_m),\nonumber
\end{align}
where 
$
\cbfI_\bolde = \begin{bmatrix}\bolde_i \bolde_i^T\otimes \mathbf{I}_n &\mathbf{0}\\ \mathbf{0}&\bolde_i \bolde_i^T\otimes \mathbf{I}_r\end{bmatrix}.
$ Then,  further expanding the terms yields
\begin{align}
\frac{\partial E_N}{\partial \widetilde{\lambda}_i}=-2vec(\boldI_p)^T&\Big(\widehat{\boldC}\otimes \boldC\Big) \sum \limits_{k=1}^N \sum \limits_{l=0}^{k}\Big( \mathbf{L}^{-1}\mathbf{K}\Big)^{l}\mathbf{L}^{-1}\Big(\bolde_i\bolde_i^T\otimes \boldI_n\Big)\nonumber\\
&
\times\big(\mathbf{L}^{-1}\mathbf{K}\big)^{k-l}\,\mathbf{L}^{-1}\,(\widehat{\boldB}\otimes\boldB) vec(\boldI_m) \nonumber\\
+2vec(\boldI_p)^T&\Big(\widehat{\boldC}\otimes \boldCr\Big) \sum \limits_{k=1}^N \sum \limits_{l=0}^{k}\Big( 
\widetilde{\mathbf{L}} ^{-1}\widetilde{\mathbf{K}}\Big)^{l}\widetilde{\mathbf{L}}^{-1}\Big(\bolde_i\bolde_i^T\otimes \boldI_r\Big) \nonumber\\
&\Big(\widetilde{\mathbf{L}}^{-1}\widetilde{\mathbf{K}}\Big)^{k-l}\,\widetilde{\mathbf{L}}^{-1}(\widehat{\boldB}\otimes\boldBr) vec(\boldI_m) 
\label{truncated_differentiate_wrt_lambda},
\end{align}
where 
$
\mathbf{L} = (-\widetilde{\mathbf{\Lambda}}\otimes \boldI_n-\boldI_r\otimes \boldA)^{-1}$,
$\mathbf{K} = \sum \limits_{j=1}^m\widehat{\boldN}_j\otimes\boldN_j$, 
$\widetilde{\mathbf{L}} =
(-\widetilde{\mathbf{\Lambda}}\otimes \boldI_r-\boldI_r\otimes \boldAr)^{-1}$, and
$\widetilde{\mathbf{K}} =\sum \limits_{j=1}^m\widehat{\boldN}_j\otimes\widetilde\boldN_j.
$
Then, the necessary condition resulting from expression (\ref{truncated_differentiate_wrt_lambda}) is
\begin{align}
vec(\boldI_p)^T\Big(\widehat{\boldC}\otimes \boldC\Big) \sum \limits_{k=1}^N \sum \limits_{l=0}^{k}\Big( \mathbf{L}^{-1}\mathbf{K}\Big)^{l}\mathbf{L}^{-1}\Big(\bolde_i\bolde_i^T\otimes \boldI_n\Big)
\times\big(\mathbf{L}^{-1}\mathbf{K}\big)^{k-l}\,\mathbf{L}^{-1}\,(\widehat{\boldB}\otimes\boldB) vec(\boldI_m)   \nonumber\\
=vec(\boldI_p)^T\Big(\widehat{\boldC}\otimes \boldCr\Big) \sum \limits_{k=1}^N \sum \limits_{l=0}^{k}\Big( 
\widetilde{\mathbf{L}} ^{-1}\widetilde{\mathbf{K}}\Big)^{l}\widetilde{\mathbf{L}}^{-1}\Big(\bolde_i\bolde_i^T\otimes \boldI_r\Big) 
\Big(\widetilde{\mathbf{L}}^{-1}\widetilde{\mathbf{K}}\Big)^{k-l}\,\widetilde{\mathbf{L}}^{-1}(\widehat{\boldB}\otimes\boldBr) vec(\boldI_m).
 \label{necessary_lambda}
\end{align}
By applying Cauchy's product rule to (\ref{necessary_lambda}), we obtain (\ref{lambda_Sylvester_condition}).
Simplifying these bloated expressions for the other derivatives requires exactly the same kinds of steps as in simplifying the derivative of $E_N$ with respect to the parameters in $\widehat{\boldC}$ and $\widetilde{\lambda}_i$, so we omit the derivations here.  The resulting expression for the derivative with respect to $\widehat{\boldB}_{i,j}$ yields
\begin{align}
vec(\boldI_p)^T&(\widehat{\boldC} \otimes \boldC) \sum \limits_{k=0}^N\Big( \mathbf{L}^{-1}\mathbf{K}\Big)^k\mathbf{L}^{-1}\Big(\bolde_i\bolde_j^T\otimes\boldB\Big)vec(\boldI_m) \nonumber\\
&=vec(\boldI_p)^T(\widehat{\boldC} \otimes \widetilde\boldC) \sum \limits_{k=0}^N\Big( \widetilde{\mathbf{L}}^{-1}\widetilde{\mathbf{K}}\Big)^k\mathbf{L}^{-1}\Big(\bolde_i\bolde_j^T\otimes\widetilde\boldB\Big)vec(\boldI_m),  \label{truncated_differentiate_wrt_b}
\end{align}
which then leads to (\ref{truncated_Sylvester_b_necessary}). 
Finally, the necessary condition resulting from the derivative of $\widehat{\boldN}_\ell(i,j)$, for $\ell=1,\ldots,m$ is
\begin{align}
&vec(\boldI_p)^T\Big(\widehat{\boldC}\otimes \boldC\Big) \sum \limits_{k=1}^N \sum \limits_{l=1}^{k}\big( \mathbf{L}^{-1}\mathbf{K}\big)^{l-1}\mathbf{L}^{-1}\Big(\bolde_i\bolde_j^T\otimes \boldN_\ell\Big)\nonumber
\big(\mathbf{L}^{-1}\mathbf{K}\big)^{k-l}\mathbf{L}^{-1}\Big(\widehat{\boldB}\otimes\boldB\Big) vec(\boldI_m)\nonumber\\
&=vec(\boldI_p)^T\Big(\widehat{\boldC}\otimes \widetilde\boldC\Big) \sum \limits_{k=1}^N \sum \limits_{l=1}^{k}\big( \widetilde{\mathbf{L}}^{-1}\widetilde{\mathbf{K}}\big)^{l-1}\widetilde{\mathbf{L}}^{-1}\Big(\bolde_i\bolde_j^T\otimes \widetilde\boldN_\ell\Big)\nonumber
\big(\widetilde{\mathbf{L}}^{-1}\widetilde{\mathbf{K}}\big)^{k-l}\widetilde{\mathbf{L}}^{-1}\Big(\widehat{\boldB}\otimes\widetilde\boldB\Big) vec(\boldI_m),\nonumber
\end{align}
which, then, can be written equivalently as (\ref{truncated_n_Sylvester_necessary}).
$\Box$
\end{document}